\documentclass[a4paper, reqno]{amsart}

\usepackage[T1]{fontenc}
\usepackage{tgadventor}

\usepackage{amsmath,amsthm,amssymb,relsize}
\usepackage[nobysame,alphabetic]{amsrefs}
\usepackage[margin=30mm]{geometry}
\usepackage{mathrsfs,mathtools,makecell}

\usepackage{booktabs, multirow, adjustbox, array}

\newcolumntype{R}[2]{%
    >{\adjustbox{angle=#1,lap=\width-(#2)}\bgroup}%
    l%
    <{\egroup}%
}

\usepackage{caption}

\usepackage{enumerate}
\usepackage{tikz}
\usetikzlibrary{
  cd,
  calc,
  positioning,
  fit,
  arrows,
  decorations.pathreplacing,
  decorations.markings,
  shapes.geometric,
  backgrounds,
  bending
}
\usepackage{tikzsymbols}
\PassOptionsToPackage{dvipsnames}{xcolor}
\definecolor{darkblue}{rgb}{0,0,0.6}

\usepackage[breaklinks, pdftex, ocgcolorlinks,colorlinks=true, citecolor=darkblue, filecolor=darkblue, linkcolor=darkblue, urlcolor=darkblue]{hyperref}
\usepackage[capitalize,noabbrev]{cleveref}

\makeatletter
\newtheorem*{rep@theorem}{\rep@title}
\newcommand{\newreptheorem}[2]{%
\newenvironment{rep#1}[1]{%
 \def\rep@title{#2 \ref{##1}}%
 \begin{rep@theorem}}%
 {\end{rep@theorem}}}
\makeatother

\newtheorem{proposition}{Proposition}[section]
\newtheorem{theorem}[proposition]{Theorem}
\newtheorem{corollary}[proposition]{Corollary}
\newtheorem{lemma}[proposition]{Lemma}

\theoremstyle{definition}
\newtheorem{definition}[proposition]{Definition}

\newtheorem{example}[proposition]{Example}

\theoremstyle{remark}
\newtheorem{remark}[proposition]{Remark}

\newtheorem*{remark*}{Remark}

\newreptheorem{theorem}{Theorem}
\newreptheorem{lemma}{Lemma}
\newreptheorem{proposition}{Proposition}
\newreptheorem{corollary}{Corollary}
\newreptheorem{question}{Question}
\numberwithin{equation}{section}

\newcommand{\lk}{\operatorname{lk}}

\newcommand{\Q}{\mathbb{Q}}
\newcommand{\N}{\mathbb{N}}

\newcommand{\R}{\mathbb{R}}
\newcommand{\Z}{\mathbb{Z}}

\newcommand{\A}{\mathcal{A}}

\newcommand{\im}{\operatorname{Im}}
\newcommand{\Id}{\operatorname{Id}}
\newcommand{\id}{\operatorname{Id}}

\newcommand{\Bl}{\mathcal{B}\ell}
\newcommand{\Qt}{{\Q[t^{\pm 1}]}}
\newcommand{\Zt}{{\Z[t^{\pm 1}]}}

\newcommand{\Ext}{\mathrm{Ext}}
\newcommand{\Hom}{\mathrm{Hom}}

\newcommand{\ol}{\overline}
\newcommand{\wt}{\widetilde}
\newcommand{\wh}{\widehat}
\newcommand{\sm}{\setminus}

\DeclareMathOperator{\Aut}{Aut}

\DeclareMathOperator{\grk}{g-rk}
\DeclareMathOperator{\coker}{coker}

\newcommand{\smfrac}[2]{\mbox{\footnotesize$\displaystyle\frac{#1}{#2}$}} 

\newcommand{\bsm}{\left(\begin{smallmatrix}}
\newcommand{\esm}{\end{smallmatrix}\right)}

\usepackage{letltxmacro}

\LetLtxMacro\Oldfootnote\footnote

\begin{document}
\title[Strongly invertible knots and equivariant slice genera]{Strongly invertible knots,  equivariant slice genera, \\ and an equivariant algebraic concordance group}

\author{Allison N.\ Miller}
\address{Department of Mathematics \& Statistics,  Swarthmore College, Swarthmore, PA USA}
\email{amille11@swarthmore.edu}

\author{Mark Powell}
\address{Department of Mathematical Sciences, Durham University, Upper Mountjoy,\newline \indent Stockton Road, Durham, United Kingdom, DH1 3LE}
\email{mark.a.powell@durham.ac.uk}


\def\subjclassname{\textup{2020} Mathematics Subject Classification}
\expandafter\let\csname subjclassname@1991\endcsname=\subjclassname

\begin{abstract}
We use the Blanchfield form to obtain a lower bound on the equivariant slice genus of a strongly invertible knot.  For our main application, let $K$ be a genus one strongly invertible slice knot with nontrivial Alexander polynomial. We show that the equivariant slice genus of an equivariant connected sum $\#^n K$ is at least $n/4$.

We also formulate an equivariant algebraic concordance group, and show that the kernel of the forgetful map to the classical algebraic concordance group is infinite rank. \end{abstract}

\maketitle

\section{Introduction}

Let $\gamma$ be a great circle in $S^3$, and let $\tau \colon S^3 \to S^3$ be the order two diffeomorphism given by the rotation with axis $\gamma$ through $\pi$ radians.
Let $K$ be a knot in $S^3$ that intersects $\gamma$ in precisely two points, and such that $\tau(K) = K$.  Then we say that $K$ is \emph{strongly invertible} with \emph{strong inversion}~$\tau$.   Note that $\tau|_K$ is necessarily orientation reversing.
Suppose that $K$ bounds a compact, oriented, locally flat surface $\Sigma$ of genus $g$ in $D^4$, such that for some extension of  $\tau$ to a locally linear involution $\wh{\tau} \colon D^4 \to D^4$, one has
that $\wh{\tau}(\Sigma) = \Sigma$.
The minimal such $g$ is called the (topological) \emph{equivariant 4-genus} or \emph{equivariant slice genus} of $(K,\tau)$, and denoted $\wt{g}_4(K,\tau)$.
A strongly invertible knot with $\wt{g}_4(K, \tau)=0$ is called \emph{equivariantly slice}.

\subsection{Lower bounds on the equivariant 4-genus}

By studying the Alexander module and the Blanchfield pairing, we derive a new lower bound for the equivariant $4$-genus, which we will explain below.  First,  we state our main application.

\begin{theorem}\label{thm:intro-arb-large}
Let $K$ be a genus one algebraically slice knot with nontrivial Alexander polynomial and strong inversion $\tau$.
Let $(K_n, \tau_n)$ be an equivariant connected sum of $n$ copies of $(K, \tau)$.
Then the equivariant 4-genus of $(K_n, \tau_n)$ is at least $n/4$. In particular, if $K$ is slice then $\wt{g}_4(K_n,\tau_n) - g_4(K_n) \to \infty$ as $n \to \infty$.
\end{theorem}

\begin{figure}[h!]
 \raisebox{-0.5\height}{\includegraphics[height=3cm]{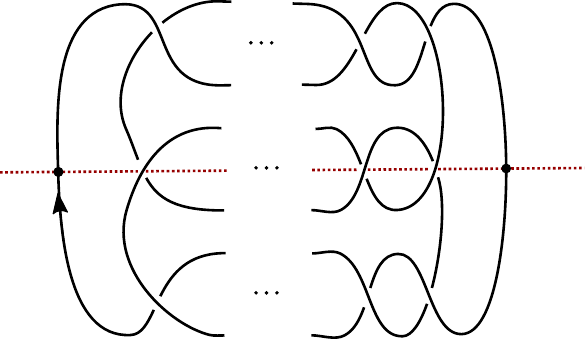}}
 \qquad
 \raisebox{-0.52\height}{\includegraphics[height=3.4cm]{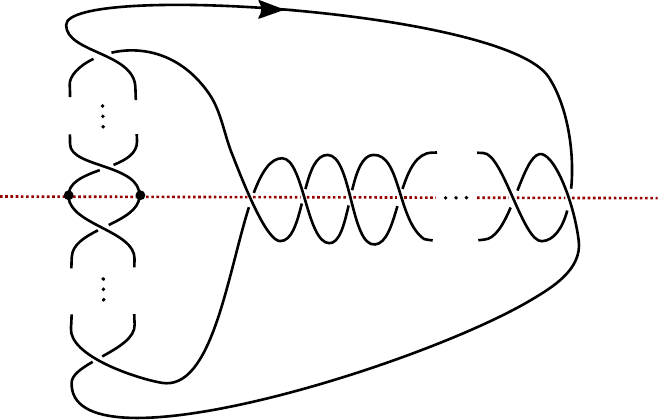}}
\caption{The pretzel knots $P(a,-a,a)$ for odd $a>1$ (left) and the generalized twist knots with continued fraction expansion $[b, b+2]^+$ for even $b>0$ (right) are strongly invertible genus one slice knots with nontrivial Alexander polynomials.}
\label{fig:examplefamilies}
\end{figure}

There are many examples of strongly invertible genus one slice knots with nontrivial Alexander polynomial;  see for example Figure~\ref{fig:examplefamilies}.
Thus the topological equivariant 4-genus of strongly invertible slice knots can be arbitrarily large.

A fixed knot can admit multiple inequivalent strong inversions,  as is the case for twist knots with even crossing number.  The equivariant connected sum is also not unique, and depends on the choice of a direction~\cite{Sakuma} (see Section~\ref{sec:equiv-conn-sum-and-equiv-conc-gp}). However Theorem~\ref{thm:intro-arb-large} holds for any choice of strong inversion on $K$ and any equivariant connected sum $K_n$,  so long as we use the same $\tau$ for each copy of $K$.

The study of strongly invertible knots up to equivariant concordance was instigated by Sakuma \cite{Sakuma}, who defined an equivariant knot concordance group of directed strongly invertible knots, and introduced the $\eta$ polynomial,  a homomorphism  from the equivariant knot concordance group to the abelian group $\Zt$.   The $\eta$ polynomial was originally defined as an obstruction to smooth equivariant concordance. However using results from \cite{FQ} one can show that it extends to an obstruction in the topological category.  Borodzik-Dai-Mallick-Stoffregen inform us that a proof of this will appear soon in work of theirs on equivariant concordance, so to avoid duplicating effort we will not provide our own proof.  It follows that the $\eta$ polynomial obstructs many strongly invertible slice knots from being equivariantly slice. But beyond this $\eta$ does not give information on the equivariant 4-genus.

Theorem~\ref{thm:intro-arb-large} gives an alternative proof of the analogous smooth result, that the smooth equivariant 4-genus can be arbitrarily large for slice knots, due to Dai-Mallick-Stoffregen~\cite{DMS}, and proven using knot Floer homology.
Our methods do not recover their specific examples, but tend to require significantly easier computations, as evidenced by the fact that Theorem~\ref{thm:intro-arb-large} applies to a large class of strongly invertible knots.  In \cite{DMS} they also consider the \emph{isotopy-equivariant 4-genus}, where one relaxes the condition that $\wh{\tau}(\Sigma) = \Sigma$ to instead to require that $\wh{\tau}(\Sigma)$ is ambiently isotopic rel.\ boundary to $\Sigma$.  Our lower bounds extend to this setting with identical proofs.

There has been further significant recent interest in equivariant concordance of strongly invertible knots, including by Dai-Hedden-Mallick~\cite{DHM},  Boyle-Issa~\cite{Boyle-Issa},  Alfieri-Boyle~\cite{Alfieri-Boyle}, and 
Di Prisa~\cite{DiPrisa}.

Theorem~\ref{thm:intro-arb-large} is a consequence of our main obstruction theorem, which reads as follows. The \emph{generating rank} of a $\Qt$-module $Q$, denoted $\grk Q$, is by definition the number of elements in a generating set of minimal cardinality. The (rational) \emph{Alexander module} of an oriented knot $K$ is the first homology $\A(K):= H_1(E_K;\Qt)$ of the infinite cyclic cover of the knot exterior $E_K$. This admits a nonsingular, Hermitian, sesquilinear \emph{Blanchfield pairing}~\cite{Blanchfield} $\Bl_K \colon \A(K) \times \A(K) \to \Q(t)/\Qt$, whose definition we will recall in detail in Section~\ref{section:Bl-form}.

\begin{theorem}\label{thm-main-obstruction-intro}
  Let $(K,\tau)$ be a strongly invertible knot. Let $k$ be the maximal generating rank of any submodule $P$ of $\A(K)$ satisfying
$\Bl_K(x, y)=0=\Bl_K(x, \tau_*(y))$
   for all $x,y \in P$.  Then
\[\wt{g}_4(K,\tau) \geq \frac{\grk \A(K) - 2k}{4}.\]
\end{theorem}

In order to apply this lower bound to a given knot, one only needs to make a relatively straightforward computation of the Blanchfield pairing, which can be done in terms of a Seifert matrix using~\cite{Kearton,Friedl-Powell-Moscow}.  To prove Theorem \ref{thm:intro-arb-large}, we compute that  $k=0$  and $\grk \A(K) =n$ when $K=\#^n J$ and $J$ is a  genus one algebraically slice knot with $\Delta_J$ nontrivial.

\subsection{An equivariant algebraic knot concordance group}

A \emph{direction} for a strongly invertible knot $(K,\tau)$ consists of a choice an orientation of the great circle $\gamma$ and a choice of connected component of $\gamma \sm K$.   The set of directed, strongly invertible knots admits a well-defined connected sum, with respect to which it forms a group if we quotient by the relation of equivariant concordance. Here two strongly invertible knots $(K_1,\tau_1)$ and $(K_2,\tau_2)$   are \emph{equivariantly concordant} if there is a locally flat concordance $A \subseteq S^3 \times I$ between $K_1$ and $K_2$ together with an extension of $\tau_1$ and $\tau_2$ to an involution $\wh{\tau} \colon S^3 \times I \to S^3 \times I$ with $\wh\tau(A) =A$, and such that the directions are preserved. We give more details in Section~\ref{sec:equiv-conn-sum-and-equiv-conc-gp}.

Taking the Blanchfield form $\Bl_K^{\Z}$ of a knot $K$ gives rise to a homomorphism from the knot concordance group to the algebraic knot concordance group, $\Bl \colon \mathcal{C} \to \mathcal{AC}$. The latter is the Witt group of abstract Blanchfield forms, which is isomorphic to the possibly more familiar Witt group of Seifert forms. The analogous homomorphism for odd high  dimensional knots $S^{2k-1} \subseteq S^{2k+1}$ is an isomorphism for $k \geq 2$ \cite{Levine:1969-1, Cappell-Shaneson-top-knots}.
For $k=1$ the algebraic concordance group has been the framework for deeper investigation of $\mathcal{C}$. See for example~\cite{CassonGordon, Jiang, Endo, Liv-2002-b, Liv-2002-a, COT1,Cochran-Harvey-Leidy:2009-1, Liv-2010, CHL-fractal, Powell-second-order, Franklin, OSS-upsilon, Cha-2021}.

The Blanchfield form interpretation of the algebraic concordance group lends itself to generalisation to the equivariant setting.
Let $\mathcal{C}^{SI}$ denote the equivariant concordance group of strongly invertible knots.  We define an \emph{equivariant algebraic concordance group} $\mathcal{AC}^{SI}$ by considering a Witt group of  abstract Blanchfield forms $(H,\Bl)$ endowed with an anti-isometry $\tau \colon H \to H$, and requiring metabolisers to be $\tau$-invariant. We give the detailed definition of $\mathcal{AC}^{SI}$ and a homomorphism $\Psi \colon \mathcal{C}^{SI} \to \mathcal{AC}^{SI}$ in Section~\ref{section:equiv-alg-conc}.   This fits into a commutative diagram, where the vertical maps forget the inversion and the horizontal maps pass from geometry to algebra.
\[\begin{tikzcd}
  \mathcal{C}^{SI} \arrow[r,"\Psi"] \arrow[d,"F"] & \mathcal{AC}^{SI} \arrow[d] \\
  \mathcal{C} \arrow[r] & \mathcal{AC}.
\end{tikzcd}\]
The bottom homomorphism is surjective~\cite{Levine:1969-1} and has kernel of infinite rank~\cite{Jiang}.

\begin{theorem}\label{thm:subgroup-of-ker-F}
  There is a subgroup of $\ker( F \colon \mathcal{C}^{SI} \to \mathcal{C})$ of infinite rank, and whose image in $\mathcal{AC}^{SI}$ is infinite rank.
\end{theorem}

\begin{remark}\label{remark-what-do-we-know}
  What else do we know about the maps in the square above?
\begin{enumerate}
  \item It follows from \cite{Liv-83} that $F$ is not surjective. If it were, every knot would be concordant to a reversible knot, but a knot that is concordant to a reversible knot is concordant to its own reverse, and Livingston found knots not concordance to their own reverses.
 \item We do not know whether $\Psi$ is surjective, nor whether the forgetful map $\mathcal{AC}^{SI} \to \mathcal{AC}$ is surjective.  However some evidence towards the surjectivity of $\Psi$ was given by Sakai,  who showed~\cite{Sakai} that every Alexander polynomial is realized by a strongly invertible knot.
\item\label{item:sakuma} Sakuma's $\eta$-invariant~\cite{Sakuma} was already an effective way to obstruct knots from being equivariantly slice, and Sakuma used it to show that $\ker F$ is nontrivial, for example by showing that the Stevedore knot is not equivariantly slice.  Moreover he showed that for $K$ the untwisted Whitehead double of the trefoil and figure eight knots, $\eta(K) \neq 0$, and therefore $\ker F \cap \ker \Psi$ is nontrivial. Note that these examples are isotopy-equivariantly slice by \cite{Conway-Powell-HR-discs}, since they have Alexander polynomial one.
\item All of $\mathcal{AC}^{SI}$, $\mathcal{AC}$, and $\mathcal{C}$ are abelian, whereas $\mathcal{C}^{SI}$ is not~\cite{DiPrisa}. So the nontrivial commutators found by Di Prisa also lie in $\ker F \cap \ker \Psi$.
\end{enumerate}
\end{remark}

Levine and Stoltzfus~\cite{Levine:1969-1, Stoltzfus} algebraically computed that $\mathcal{AC} \cong \Z^{\infty} \oplus (\Z/2\Z)^{\infty} \oplus (\Z/4\Z)^{\infty}$.  We aim to analyse the isomorphism type of $\mathcal{AC}^{SI}$ in future work.

In our proof of Theorem~\ref{thm:subgroup-of-ker-F}, we use genus one knots to exhibit the claimed subgroup of $\ker F$.
The proof also shows the following result, which we think worth emphasising.

\begin{corollary}\label{cor:notslice}
Let $K$ be a genus one strongly invertible knot with Alexander polynomial nontrivial. Then $K$ is not equivariantly slice.
\end{corollary}

This is also formally a consequence of Theorem~\ref{thm:intro-arb-large}, although we prove the corollary before we start the proof of Theorem~\ref{thm:intro-arb-large}.

\subsubsection*{Organisation of the paper}

In Section~\ref{section:Bl-form} we recall the Alexander module and the Blanchfield form of a knot, and we show that a strong involution $\tau$ induces an anti-isometry of the Blanchfield form. We also recall the definition of equivariant connected sum, and consider its effect on this data.
In Section~\ref{section-Bl-computations} we perform some detailed computations of Blanchfield forms on some key examples.
In Section~\ref{section:equiv-alg-conc} we show that the Blanchfield form of an equivariantly slice strongly invertible knot has an equivariant metaboliser.  We use this observation to motivate the definition of an equivariant algebraic concordance group $\mathcal{AC}^{SI}$, and we then prove Theorem~\ref{thm:subgroup-of-ker-F} and Corollary~\ref{cor:notslice}.  We also give an infinite family of amphichiral knots that are infinite order in~$\mathcal{AC}^{SI}$.
Finally, Section~\ref{section:lower-bound} contains the proof of the lower bound in Theorem~\ref{thm-main-obstruction-intro}, and then by combining this theorem with the computations in Section~\ref{section-Bl-computations}, we deduce Theorem~\ref{thm:intro-arb-large}.

\subsubsection*{Acknowledgements}

The authors are grateful to Irving Dai, Abhishek Mallick, and Matthew Stoffregen for interesting conversations with MP about their work~\cite{DMS}, which motivated this project.
MP is also grateful to Maciej Borodzik and Wojciech Politarczyk for interesting conversations about strongly invertible knots.
The authors thank Chuck Livingston for helpful comments on a draft of this piece.

MP was partially supported by EPSRC New Investigator grant EP/T028335/1 and EPSRC New Horizons grant EP/V04821X/1.

\section{Alexander modules, Blanchfield forms,  equivariant connected sum,  and equivariant concordance}\label{section:Bl-form}

In this section we recall the equivariant connected sum of strongly invertible knots, and we deduce algebraic conclusions, on the level of Alexander modules and Blanchfield pairings, from the existence of a strongly invertible slice disc.
Throughout this section and the remainder of the article we write $\Lambda := \Qt$.

\begin{definition}
Given a finitely generated $\Lambda$-module $U$,  we define the following notions.
\begin{enumerate}
\item The order of $U$ is denoted by $|U|$,  and is an element of $\Lambda$ well-defined up to multiplication by units of $\Lambda$.
\item The $\Lambda$-module  $\ol{U}$ setwise agrees with $U$ and has $\Lambda$-action defined by $p(t) \cdot_{\ol{U}} u= \ol{p(t)} \cdot_U u$ for all $p(t) \in \Lambda$ and $u \in U$,  where $\ol{\, \cdot \,}$ is the $\Q$-linear involution on $\Lambda$ sending $t^k$ to $t^{-k}$ for all $k \in \mathbb{Z}$.
\end{enumerate}
\end{definition}

\subsection{The involution induced on the Alexander module}

Let $K$ be an oriented knot in $S^3$ and let $E_K := S^3 \sm \nu K$ denote the exterior of $K$. Let $M_K := S^3_0(K)$ denote the result of 0-framed surgery on $S^3$ along $K$.
Let $\mu_K$ be an oriented meridian for $K$ and let $\lambda_K$ be a 0-framed oriented longitude.
Requiring that $\mu_K$ maps to $1 \in \Z$ determines surjections factoring through the Hurewicz maps:
\begin{align*}
\pi_1(E_K) & \to H_1(E_K;\Z) \cong \Z\cong \langle t \rangle,  \\
\pi_1(M_K) &\to H_1(M_K;\Z) \cong \Z\cong \langle t \rangle.
\end{align*}
These in turn determine coefficient systems for twisted homology:
\begin{align*}
H_i(E_K;\Lambda)& :=H_i(\Lambda \otimes_{\Z[\pi_1(E_K)]} C_*(\wt{E}_K)), \\
H_i(M_K;\Lambda)& :=H_i(\Lambda \otimes_{\Z[\pi_1(M_K)]} C_*(\wt{M}_K)).
\end{align*}
 One makes analogous constructions for $\Zt$ in place of $\Lambda = \Qt$.

\begin{definition}
The (rational) \emph{Alexander module} of $K$ is the homology $\A(K) := H_1(E_K;\Lambda)$.
The inclusion-induced map $H_1(E_K; \Lambda) \to H_1(M_K; \Lambda)$ is an isomorphism,
since the latter is the quotient of $H_1(E_K;\Lambda)$ by the class represented by the 0-framed longitude of $K$, which is the trivial class because any Seifert surface for $K$ exhibits the longitude as a double commutator in $\pi_1(E_K)$.

The \emph{integral Alexander module} of $K$ is $\A^{\Z}(K) := H_1(E_K;\Zt)$, or equivalently $H_1(M_K;\Zt)$, for the same reason as above.
\end{definition}

A strong involution determines some additional algebraic structure on the Alexander module, of the following type.

\begin{definition}\label{defn:anti-automorphism}
Let $U$ be a $\Lambda$-module.
A $\Lambda$-module isomorphism $f \colon U \to \ol{U}$ is called an \emph{anti-automorphism}. That is, $f$ is a $\Q$-linear bijection  such that
\[f(t^k \cdot_U x)= t^{k} \cdot_{\ol{U}} f(x)= t^{-k} \cdot_U f(x)\]
for all $k \in \mathbb{Z}$ and $x \in U$.  An analogous definition holds for $\Zt$-modules in place of $\Lambda$-modules.
\end{definition}

Let $K$ be a strongly invertible knot with involution $\tau$.
Restricting $\tau$ to $E_K$ gives a function $\tau_{K} \colon E_K \to E_{K^r}$ that sends $\mu_K$ to $\mu_{K^r}$ and $\lambda_K \to \lambda_{K^r}$.
For every knot $J$,  the identity map on $S^3$ restricts to a function $\rho_J \colon E_{J} \to E_{J^r}$ that sends $\mu_J$ to $\mu_{J}^{-1}$ and $\lambda_J$ to $\lambda_J^{-1}$.

Therefore we can consider the composition
\[\rho_{K}^{-1}  \circ \tau \colon E_K \to E_K, \]
an orientation preserving homeomorphism of $E_K$ that squares to the identity map and satisfies   $\mu_K \mapsto \mu_K^{-1}$,  and $\lambda_K \mapsto \lambda_K^{-1}$.

Further, the map $\rho_{K}^{-1}  \circ \tau$ induces an anti-automorphism of $\A(K)$, as in Definition~\ref{defn:anti-automorphism},  which by a mild abuse of notation we refer to as $\tau_* \colon \A(K) \to \A(K)$.

\begin{definition}
Given a strongly invertible knot $(K, \tau)$,  let
\[ \tau_*:=(\rho_K^{-1} \circ \tau)_* \colon \A(K) \to \A(K)\]
be defined as above.  We call the anti-automorphism $\tau_*$ the \textit{inversion-induced map on the Alexander module}.
There is an analogous map $\tau_* \colon \A^{\Z}(K) \to \A^{\Z}(K)$.
\end{definition}

\begin{example}
We give some examples of inversion-induced maps.
\begin{enumerate}[(i)]
\item  For $K= 6_1$ the Stevedore knot, i.e.\ $b=2$ on the right of Figure~\ref{fig:examplefamilies}, we have that $\A(K) \cong \Lambda/(2t-5 + 2t^{-1})$, and $\tau_*(p(t)) = -p(t^{-1})$.
\item Let $K=4_1$ be the figure-eight knot,  with involution as indicated in Figure~\ref{fig:amphi}.  Then $\A(K) \cong \Lambda/ (t-3+t^{-1})$ and $\tau_*(q(t))=q(t^{-1})$.
\item Let $K = J \# J^r$ and let $\tau$ be the involution which switches the two factors. Then $\A(K)\cong \A(J) \oplus \A(J^r)$ and $\tau_*(x,y) = (y,x)$.
One can verify that for any $(x,y) \in A(J) \oplus \A(J^r)$ one has
\begin{align*}
\tau_*(t_{J \#J^r} \cdot(x,y))
= \tau_*(t_J \cdot x, t_{J^r} \cdot y)
=(t_{J^r} \cdot y,  t_J \cdot x)
= (t_{J}^{-1} \cdot y,  t_{J^r}^{-1} \cdot x)
=t_{J \# J^r}^{-1} \cdot \tau_*(x,y),
\end{align*}
so $\tau_*$ is indeed an anti-automorphism.
\end{enumerate}
\end{example}

\subsection{The Blanchfield pairing}\label{sec:Bl-pairing}

As above let $M_K$ be the result of zero-framed surgery on $S^3$ along a knot $K$.
We consider the sequence of isomorphisms
\[ \Theta \colon H_1(M_K;  \Lambda) \xrightarrow{\cong} H^2(M_K;\Lambda) \xrightarrow{\cong} H^1(M_K;\Q(t)/\Lambda) \xrightarrow{\cong} \Hom_{\Lambda}(H_1(M_K;\Lambda),\Q(t)/\Lambda).\]
These maps are given respectively by the inverse of Poincar\'{e} duality, the inverse of a Bockstein isomorphism corresponding to the short exact sequence of coefficients $0 \to \Lambda \to \Q(t) \to \Q(t) /\Lambda \to 0$, and the evaluation map. The Bockstein map is an isomorphism because $H_i(M_K;\Q(t)) =0$ for $i=1,2$, since $H_i(M_K;\Lambda)$ is $\Lambda$-torsion.  The evaluation map is an isomorphism by the universal coefficient theorem,  because $\Q(t)/\Lambda$ is an injective $\Lambda$-module and so $\Ext^1_{\Lambda}(H_0(M_K;\Lambda),\Q(t)/\Lambda) =0$.

\begin{definition}
  The Blanchfield pairing $\Bl_K \colon H_1(M_K;  \Lambda) \times H_1(M_K;  \Lambda) \to \Q(t)/\Lambda$ is given by $\Bl_K(x,y) := \Theta(y)(x)$.
\end{definition}

The Blanchfield pairing, originally defined in \cite{Blanchfield}, is sesquilinear, Hermitian, and nonsingular; see for example~\cite{Blanchfield-MP} for more details.  Here sesquilinear means that $\Bl_K(px,qy) = p \Bl_K(x,y) \ol{q}$ and Hermitian means that $\Bl_K(y,x) = \ol{\Bl_K(x,y)}$, for all $x,y \in H_1(M_K;\Lambda)$ and for all $p,q \in \Lambda$.
The analogous definition applies with $\Zt$ replacing $\Lambda$, giving rise to the integral Blanchfield form $\Bl^{\Z}_K$.  It is more work to show that the evaluation map is an isomorphism in this case, but it still holds, see e.g.~\cite{LevineKnotModules}.

Given a $\Lambda$-module $U$ and a sesquilinear, Hermitian pairing $B \colon U \times U \to \Q(t)/\Lambda$, there is an involuted pairing $\ol{B} \colon \ol{U} \times \ol{U} \to \Q(t)/\Lambda$ given by $\ol{B}(x,y) = \ol{B(x,y)}$.
The pairing $\ol{B}$ is sesquilinear in the sense that $\ol{B}(px,qy) = \ol{p}\ol{B}(x,y) q$.

\begin{definition}
Let $U$ be a $\Lambda$-module and let $B \colon U \times U \to \Q(t)/\Lambda$ be a sesquilinear, Hermitian pairing.
An anti-automorphism $f$ is called an \emph{anti-isometry} of $(U,B)$ if
\[ B(x,y)= \ol{B}(f(x), f(y))= \ol{B(f(x),  f(y))}\]
for all $x,y \in U$.  That is,  an anti-isometry induces an isometry  between $B \colon U \times U \to \Q(t)/\Lambda$ and $\ol{B} \colon \ol{U} \times \ol{U} \to \Q(t)/\Lambda$.
An analogous definition holds for $\Zt$-modules and a pairing with values in $\Q(t)/\Zt$.
\end{definition}

We saw in the previous section that $\tau_* \colon H_1(M_K;\Lambda) \to H_1(M_K;\Lambda)$ is an anti-automorphism,  and now prove the following.

\begin{proposition}\label{prop:anti-isometry}
 Let $(K,\tau)$ be a strongly invertible knot.
 The inversion-induced map on the Alexander module, $\tau_*$, induces an anti-isometry of the Blanchfield pairing $\Bl_K$. The same holds for the integral Alexander module and $\Bl^\Z_K$.
\end{proposition}

\begin{proof}
The homeomorphism $\tau \colon E_K \to E_{K^r}$ induces an isometry of Blanchfield pairings $\Bl_{K} \cong \Bl_{K^r}$.  The map  $\rho_K \colon E_{K^r} \to E_K$ identifies $\Bl_{K^r}$ with $\ol{\Bl}_K$.  Therefore $\rho_{K}^{-1}  \circ \tau$, the map which induces $\tau_* \colon \A(K) \to \A(K)$, induces an isometry of $\Bl_K$ with $\ol{\Bl}_K$, or in other words an anti-isometry of $\Bl_K$, as required.
\end{proof}

\subsection{Equivariant connected sum of knots}\label{sec:equiv-conn-sum-and-equiv-conc-gp}

We recall the definition of the connected sum of two directed strongly invertible knots, following Sakuma~\cite{Sakuma}.
A \emph{direction} on a strongly invertible knot $(K,\tau)$ is a choice of orientation of the great circle~$\gamma$, and a choice of one of the two connected component of $\gamma \sm K$.  A strongly invertible knot together with a choice of direction is called \emph{directed}. This extra data enables us to remove the indeterminacy in the definition of connected sum.

\begin{definition}[Equivariant connected sum]\label{defn:connsum}
Let $(K_1,\tau_1)$ and $(K_2, \tau_2)$ be directed strongly invertible knots,  so $\tau_i \colon (S^3, K_i) \to (S^3, K_i)$ is rotation by $\pi$ about the axis $\gamma_i$.
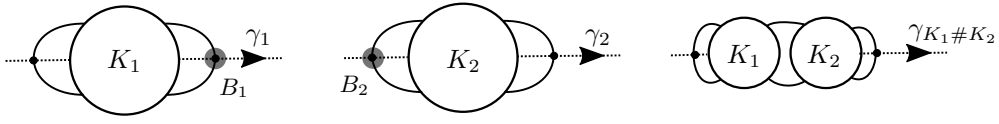
\begin{figure}[h!]
\begin{picture}(350,50)
\put(0,0){\includegraphics[height=1.5cm]{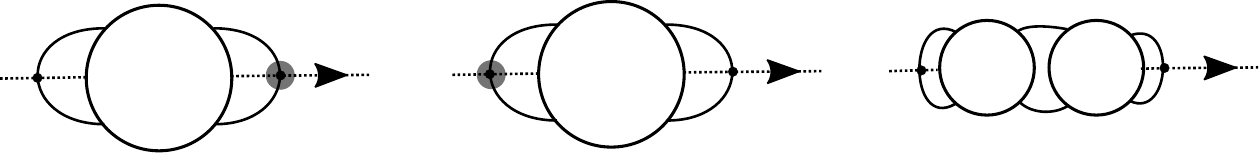}}
\put(38,18){$K_1$}
\put(90, 28){$\gamma_1$}
\put(80,8){\small $B_1$}
\put(165,18){$K_2$}
\put(217,28){$\gamma_2$}
\put(125,9){\small $B_2$}
\put(270,20){$K_1$}
\put(300,20){$K_2$}
\put(338, 30){$\gamma_{K_1\#K_2}$}
\end{picture}
\caption{Two directed strongly invertible knots (left and middle) and their connected sum (right).  For each of the three knots, the preferred component of the axis of symmetry passes through the point at infinity.  }
\label{fig:equivconnsum}
\end{figure}
 As illustrated in Figure~\ref{fig:equivconnsum}, for  $i=1,2$ let $B_i$ be a small neighborhood of one of the intersection points of $K_i$ with $\gamma_i$.  For $K_1$, use the intersection point that lies at the \emph{start} of the chosen connected component of $\gamma_1 \sm K_1$. For $K_2$, use the intersection point that lies at the \emph{end} of the chosen connected component of $\gamma_2 \sm K_2$.
 Arrange that $\overline{B}_i \cap K_i$ is an unknotted arc and such that $\tau_i$ restricts to a homeomorphism of pairs $(\overline{B}_i, \overline{B}_i \cap K_i)$.
 Let
 \[f\colon \partial\big(\overline{B}_1, \overline{B}_1 \cap K_1, \overline{B}_1 \cap\gamma_1\big) \to \partial \big(\overline{B}_2, \overline{B}_2 \cap K_2, \overline{B}_1 \cap\gamma_2\big)\]
be a homeomorphism of triples such that:
\begin{enumerate}[(i)]
  \item $f$ is an orientation-reversing homeomorphism of $S^3$;
  \item  $\tau_1 \circ f = \tau_2$;
  \item the point of $\ol{B}_1 \cap \gamma_1$ at which the orientation of $\gamma_1$ points into $B_1$ is identified with the point of $\ol{B}_1 \cap \gamma_2$ at which the orientation of $\gamma_2$ points out of $B_2$.
\end{enumerate}
 Then there is a homeomorphism of triples
\[ \big( (S^3, K_1, \gamma_1) \smallsetminus (B_1,  B_1 \cap K_1, B_1 \cap \gamma_1) \big) \cup_f \big((S^3, K_2, \gamma_2) \smallsetminus (B_2,  B_2 \cap K_2, B_2 \cap \gamma_2)\big) \cong \big(S^3, K_1 \# K_2, \gamma_1 \# \gamma_2\big)\]  which defines the equivariant connected sum $K_1 \# K_2$.  This comes with a strong involution $\tau$ obtained from gluing $\tau_1$ and $\tau_2$, with fixed set $\gamma_1 \# \gamma_2$ and such that $\tau(K_1 \# K_2) = K_1 \# K_2$.

To define the direction on the connected sum, we take the orientation on $\gamma_1 \# \gamma_2$ induced by the orientations of $\gamma_1$ and $\gamma_2$, and we take the connected component of $\gamma_1 \# \gamma_2$ which contains the original preferred components of $\gamma_1$ and $\gamma_2$ (minus $B_1 \cap \gamma_1$ and $B_2 \cap \gamma_2$, respectively).

We call $(K, \tau)$ the \textit{equivariant connected sum} of $(K_1, \tau_1)$ and $(K_2, \tau_2)$.  Sakuma~\cite[\S 1]{Sakuma} proved that the equivariant isotopy class of $K_1 \# K_2$ does not depend on the choice of $f$ satisfying the above conditions.
\end{definition}

\begin{remark}
In particular note that we did not fix an orientation on $K_1$ nor on $K_2$.    Strongly invertible knots are reversible, so $K_1 \# K_2$ is isotopic to $K_1 \# K_2^r$.  As indicated in \cite[Figure~1.2]{Sakuma}, the two knots are moreover equivariantly isotopic, and so it is not necessary to choose orientations on the~$K_i$.
\end{remark}

\begin{definition}[Equivariant concordance]
 Let $(K_0,\tau_0)$ and $(K_1,\tau_1)$ be directed strongly invertible knots with axes $\gamma_0$ and $\gamma_1$ respectively.
\begin{enumerate}[(a)]
  \item  Suppose that there is a concordance between $K_0$ and $K_1$, i.e.\ there is a locally flat embedding $c \colon S^1 \times [0,1] \to S^3 \times [0,1]$ with $c(S^1 \times \{i\})= K_i$ for $i=0,1$. The image $C:= c(S^1 \times [0,1]$ is a proper submanifold of $S^3 \times [0,1]$.
  \item  Suppose also that there is an involution of $(S^3 \times [0,1],C)$ extending $\tau_0$ and $\tau_1$, that is an order two locally linear  homeomorphism $\wh{\tau} \colon S^3 \times [0,1] \to S^3 \times [0,1]$ such that $\wh{\tau}|_{S^3 \times \{i\}} = \tau_i \colon S^3 \times \{i\} \to S^3 \times \{i\}$ for $i=0,1$, and such that $\wh{\tau}(C) =C$.
  \item Let $A$ be the set of fixed points of $\wh{\tau}$. By Remark~\ref{remark:smith-thy} below, $A$ is a locally flat concordance between $\gamma_0$ and $\gamma_1$. Suppose that the chosen  connected components of $\gamma_0 \sm K_0$ and $\gamma_1 \sm K_1$ lie in the same connected component of  $A \sm C$, and that the orientations of $\gamma_0$ and $\gamma_1$ induce opposite orientations on $A$.
  \end{enumerate}
Then we say that $(K_0,\tau_0)$ and $(K_1,\tau_1)$ are \emph{directed equivariantly concordant}.
 \end{definition}

\begin{remark}\label{remark:smith-thy}%
 We explain why the fixed point set $A$  of $\Z/2$ acting on $S^3 \times [0,1]$ via the involution~$\wh{\tau}$ is an annulus.  Since $\wh{\tau}$ is locally linear, the fixed set $A$ is a submanifold of locally constant dimension. Let $(W,  f)= (S^3 \times I, \wh{\tau})\cup (D^4, \sigma)$,  where $\sigma$ is the standard extension of $\tau$ to $D^4$.

 Smith's theorem on finite group actions~\cite{Smith-I,Smith-II}  (see also~\cite[\textsection III,~Theorem~5.2]{Bredon-trans-gps} for a modern treatment) implies that the fixed set of $(W,f)$ is a $\Z/2$-homology ball, and in particular is connected. Since the fixed set restricts to $\gamma_0$ in the boundary $S^3$, it is a connected surface with boundary $\gamma_0$, and since it is a $\Z/2$ homology ball it must be homeomorphic to $D^2$.  Remove $(D^4, \sigma)$, and note that the fixed set of $\sigma$ is also a disc, with boundary $\gamma_1$, to see that the fixed set $A$ of $\wh{\tau}$ is indeed an annulus with $\partial A = \gamma_0 \cup -\gamma_1$.
\end{remark}

Now that we have a well-defined notion of equivariant connected sum and equivariant concordance, we can define the equivariant concordance group.
See also \cite{Sakuma}, and e.g.\ \cite[Section~2]{Boyle-Issa} and  \cite[Section~2.1]{DMS}.

\begin{definition}[Equivariant concordance group]
  The set of directed equivariant concordance classes of directed strongly invertible knots forms a group under equivariant connected sum.
  The inverse of the directed strongly invertible knot $(K,\tau)$ is the knot obtained by reversing the orientations of $S^3$, with the direction given by reversing the orientation of $\gamma$ and keeping the same preferred component.
  We denote the \emph{equivariant concordance group} by $\mathcal{C}^{SI}$.
\end{definition}

As explained in the upcoming proposition, the choice of directions do not affect the Alexander module nor the Blanchfield form of an equivariant connected sum, and therefore while they are necessary in order to define $\mathcal{C}^{SI}$, we will not need to focus on them in the rest of the paper.

\begin{proposition}\label{prop:connected-sum-well-behaved}
  Let  $(K_1,\tau_1)$ and $(K_2,\tau_2)$ be strongly invertible knots.  For any choice of directions on $K_1$ and $K_2$,
  the Alexander module and Blanchfield pairing of the equivariant sum $(K_1,\tau_1) \# (K_2,\tau_2)$
 is the direct sum $(\A(K_1) \oplus \A(K_2),\Bl_{K_1} \oplus \Bl_{K_2})$,
 and with respect to this identification the induced involution is $\left(\tau_{K_1 \# K_2}\right)_* =\left( \tau_{K_1}\right)_* \oplus \left(\tau_{K_2}\right)_*$.  The same holds for the integral versions $\A^\Z(K)$ and $\Bl^\Z_K$.
\end{proposition}

\begin{proof}
The exterior of $K_1 \# K_2$ can be obtained by gluing the exteriors of $K_1$ and $K_2$ along a thickened oriented meridian for each (or, to use the perspective of Definition~\ref{defn:connsum},   gluing the exterior of the knotted arc for $K_1$ to that of the knotted arc for $K_2$).
A thickened meridian $\mu$ has $H_i(\mu;\Lambda) =0$ for $i \geq 1$, and $H_0(\mu;\Lambda) \cong \Q$ so the Alexander modules add by a Mayer-Vietoris argument. Since $\tau_{K_1 \# K_2}$ is defined by gluing  $\tau_1$ on $E_{K_1}$ and $\tau_2$ on $E_{K_2}$, it follows that it induces  $\left( \tau_{K_1}\right)_* \oplus \left(\tau_{K_2}\right)_*$ on $\A(K_1) \oplus \A(K_2) \cong \A(K_1 \# K_2)$.  It is well-known that the Blanchfield pairing of a connected sum is the direct sum as claimed. For example one can see this using the fact that the Seifert forms  add in this way, and that the Blanchfield pairing can be computed using the Seifert pairing~\cite{Kearton} (see also \cite{Friedl-Powell-Moscow}). Alternatively one can apply~\cite{FLNP}.
\end{proof}

\section{Computations of the Blanchfield pairing}\label{section-Bl-computations}

In this section we explicitly compute the Blanchfield pairing for specific families of strongly invertible knots, in particular for every genus one algebraically slice knot.   We will make use of these computations in the proofs of our main results.  To avoid interrupting the arguments later, and to be able to appeal to these computations in Sections~\ref{section:equiv-alg-conc} and \ref{section:lower-bound},  we collect these computations first  here.

The following proposition can be deduced by combining ~\cite[Theorems~1.3~and~1.4]{Friedl-Powell-Moscow}, and by passing from $\Zt$ to $\Qt$ coefficients.

\begin{proposition}[Friedl-Powell]\label{prop:alexandblanch}
Let $F$ be a Seifert surface for a knot $K$ with a collection of simple closed curves $\alpha_1,\dots, \alpha_{2g}$ on $F$ that form a basis for $H_1(F; \Z)$, and let $A$ be the corresponding Seifert matrix.
Let $\beta_1,\dots, \beta_{2g}$ be a dual basis for $H_1(S^3, F; \Z)$, i.e.\  a basis such that $\lk(\alpha_i, \beta_j)= \delta_{i,j}$.
Using the standard decomposition of  $(S^3 \smallsetminus \nu(K))_\infty= \cup_{j=-\infty}^{\infty}(S^3 \smallsetminus \nu(F))_j$,  let the homology class of the unique lift of $\beta_i$ to $(S^3 \smallsetminus \nu(F))_0$ be denoted by $b_i$. Then the map $p \colon ( \Q[t^{\pm1}])^{2g} \to \mathcal{A}(K)$ given by $p(x_1, \dots, x_{2g})= \sum_{i=1}^{2g} x_i b_i$ is a surjective map with kernel $(tA-A^T)  ( \Q[t^{\pm1}])^{2g}$.  Moreover, for $x, y \in \Q[t^{\pm1}]^{2g}$  the rational Blanchfield pairing is given by
\[ \Bl(p(x), p(y))= (t-1)x^T (A-tA^T)^{-1} \ol{y},\]
where $\bar{\cdot}$ is the component-wise extension of the $\Q$-linear involution on $\Q[t^{\pm1}]$ sending $t^i$ to $t^{-i}$.
\end{proposition}

We will use the following elementary fact to verify that certain elements of $\Q(t)/\Qt$ are nonzero.

\begin{lemma}\label{lemma:elem-fact-coprime}
  Let $p,q \in \Qt$ be coprime. Then $\smfrac{a}{p} + \smfrac{b}{q} \in \Qt \subseteq \Q(t)$ if and only if both $\smfrac{a}{p}$ and $\smfrac{b}{q}$ belong to $\Qt$.
\end{lemma}

\begin{proof}
  The if direction is trivial.  If $\smfrac{a}{p} + \smfrac{b}{q} \in \Qt$ then $\smfrac{aq+bp}{pq}\in \Qt$, so $p \mid aq+bp$, which implies $p\mid aq$. So $p\mid a$ since $p$ and $q$ are coprime. Therefore $\smfrac{a}{p} \in \Qt$. By symmetry this suffices.
\end{proof}

\begin{example}[The knot $9_{46}$]\label{example:946}
\begin{figure}[h!]
\begin{picture}(200,100)
\put(0,0){\includegraphics[height=4cm]{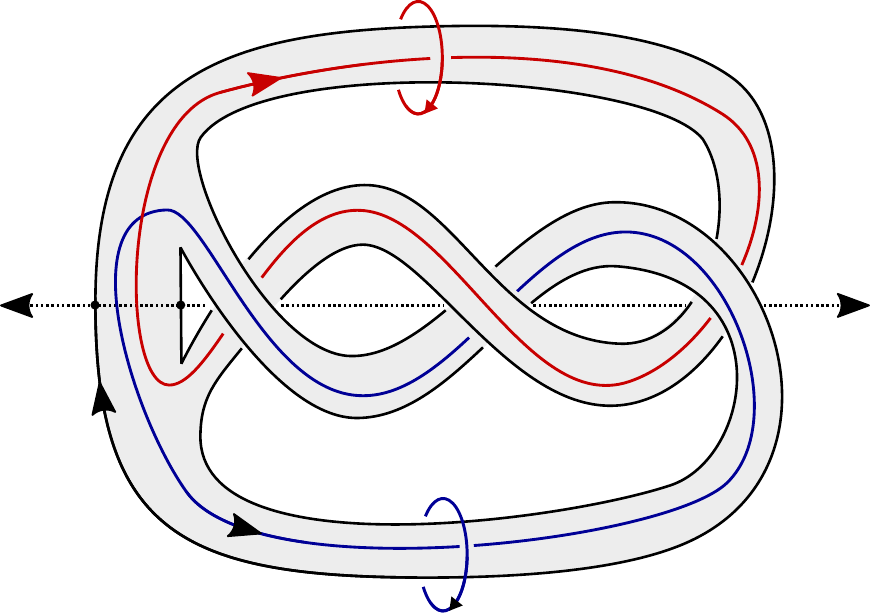}}
\put(148,72){$\alpha_1$}
\put(148, 32){$\alpha_2$}
\put(78,82){$\beta_1$}
\put(78,24){$\beta_2$}
\end{picture}
\caption{The knot $K=9_{46}$ with an axis of inversion and a Seifert surface.}
\label{fig:946}
\end{figure}
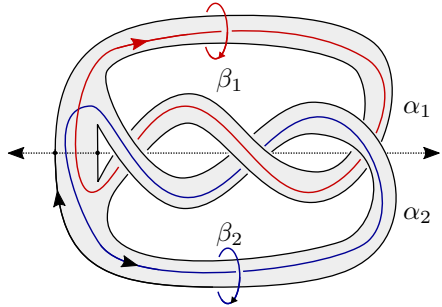
The knot $9_{46}$ is shown in Figure~\ref{fig:946}.
Denote the depicted generators for $H_1(F)$ by $\alpha_1$  and $\alpha_2$  and the depicted dual generating set for $H_1(S^3 \smallsetminus F)$ by $\beta_1$ and $\beta_2$.  The Alexander matrix for $F$ with respect to this basis is given by $A= \bsm 0 & 2 \\ 1 & 0 \esm $.
By Proposition~\ref{prop:alexandblanch},  we therefore have that
\[ \mathcal{A}(K) \cong \Q[t^{\pm1}] / \langle t-2 \rangle \oplus \Q[t^{\pm1}] / \langle 2t-1 \rangle,\]
where the first summand is generated by $b_1$ and the second summand by $b_2$.
Observe that since $\tau(\beta_1)= \beta_2$, and $\tau(\beta_2)=\beta_1$  we have that $\tau_* \colon  \mathcal{A}(K) \to  \mathcal{A}(K)$ sends $b_1$ to $ b_2$ and $b_2$ to $b_1$.  More precisely, for any $p_1(t), p_2(t) \in \Q[t^{\pm1}]$, we have that
\begin{align*}
\tau_*(p_1(t) b_1+ p_2(t) b_2)= p_2(t^{-1}) b_1 +p_1(t^{-1}) b_2.
\end{align*}

Note that every element $x$ of $\mathcal{A}(K)$ can be written as $x=c_1 b_1+ c_2 b_2$ for some $c_1, c_2 \in \mathbb{Q}$, since as abelian groups $\Qt/(t-2) \cong \Q \cong\Qt/(2t-1)$.
We now compute $\Bl(x, \tau_*(x))$ using Proposition~\ref{prop:alexandblanch}:
\begin{align*}
\Bl(x, \tau_*(x))&= \Bl(c_1b_1+c_2b_2,  c_2b_1+c_1b_2)\\
&= (t-1) \begin{pmatrix} c_1 & c_2 \end{pmatrix} \begin{pmatrix} 0 & 2-t \\ 1-2t & 0 \end{pmatrix}^{-1} \begin{pmatrix} c_2 \\ c_1 \end{pmatrix} \\
&= \smfrac{-(t-1)}{(2t-1)(t-2)} \begin{pmatrix} c_1 & c_2 \end{pmatrix} \begin{pmatrix} 0 & t-2 \\ 2t-1 & 0 \end{pmatrix} \begin{pmatrix} c_2 \\ c_1 \end{pmatrix} \\
&= -(t-1) \Big( \smfrac{(c_1)^2}{2t-1} + \smfrac{(c_2)^2}{t-2} \Big).
\end{align*}
By the additivity of the Blanchfield pairing under connected sum,  it is straightforward to obtain the Blanchfield pairing of the connected sum $\#^n 9_{46}$. Applying Lemma~\ref{lemma:elem-fact-coprime}, we deduce that $\Bl(x, \tau_*(x))=0$ if and only if $c_1=c_2=0$, i.e.\ $x=0$.
\end{example}

We can generalize Example~\ref{example:946} to the following result.

\begin{proposition}\label{prop:genusone}
Let $K$ be a genus one algebraically slice knot with strong inversion $\tau$ and nontrivial Alexander polynomial.
For each $n \in \mathbb{N}$,  let $(K_n, \tau_n):=\#^n (K, \tau)$.
For every nonzero $x \in \A(K_n)$ we have that $\Bl(x, (\tau_n)_*(x)) \neq 0$.
\end{proposition}

\begin{proof}
We begin by constraining the action of $\tau_*$ on $\A(K)$.
Since $K$ is algebraically slice and genus one, it has some Seifert surface $F$ and a basis for $H_1(F)$ with respect to which its    Seifert matrix is $A= \bsm 0 & m+1 \\ m & \ell \esm$ for some $m,\ell \in \Z$.   By further change of basis of $H_1(F)$,  we can assume that $\ell \neq 0$.

Using $tA-A^T$ to present $\A(K)$ as in Proposition~\ref{prop:alexandblanch},  we see that $\A(K)$ is generated as a $\Q[t^{\pm1}]$ module by $b_1,b_2$ subject to the relations
\begin{align*}
0&=(mt-(m+1))b_2 \text{ and }\\
0&=((m+1)t-m)b_1+\ell(t-1)b_2.
\end{align*}
Adding $-\smfrac{\ell}{m}$ times the first equation to the second and solving for $b_2$ gives us that
\[b_2= \smfrac{m}{\ell} ((m+1)t-m))b_1,\]
and hence that $\A(K)$ is cyclic with generator $b_1$.  Since the order of $\A(K)$ is exactly the Alexander polynomial, which is given by $\det(tA-A^T)$,  we obtain that
\[\A(K) \cong \Q[t^{\pm1}]/ \langle (mt-(m+1))((m+1)t-m) \rangle,\]
with $b_1$ as a generator.  Since $K$ has nontrivial Alexander polynomial we know that $m, (m+1) \neq 0$.
Let $y_1 := ((m+1)t-m) b_1$ and $y_2 :=(mt-(m+1))b_1$.  Observe that $ty_1= \smfrac{m+1}{m} y_1$ and $ty_2= \smfrac{m}{m+1} y_2$.

Now recall that $\tau_* \colon \A(K) \to \A(K)$ is a $\Q$-linear map satisfying $\tau_*(t x)= t^{-1} \tau_*(x)$ and $\tau_*(\tau_*(x))=x$ for all $x \in \A(K)$.
Since $y_1$ and $y_2$ generate $\A(K)$ as a $\Q$-module,  we can write $\tau_*(y_1)=c_1y_1+c_2y_2$ and $\tau_*(y_2)=d_1y_1+d_2y_2$ for some $c_1,c_2,d_1,d_2 \in \Q$.
Observe that
\[\tau_*(ty_1)= \tau_*\Big( \smfrac{m+1}{m} y_1\Big)= \smfrac{m+1}{m} \tau_*(y_1)= \smfrac{c_1(m+1)}{m} y_1 + \smfrac{c_2(m+1)}{m} y_2
\]
and
\[ t^{-1} \tau_*(y_1)= c_1t^{-1} y_1+c_2t^{-1}y_2
= \smfrac{c_1m}{m+1} y_1 + \smfrac{c_2(m+1)}{m} y_2.
\]
Since $\tau_*(ty_1)=t^{-1}\tau_*(y_1),$, it follows that $c_1=0$,  and an analogous argument using $\tau_*(ty_2)=t^{-1} \tau_*(y_2)$ shows that $d_2=0$ as well.
Since $(\tau_*)^2= \id$,  we can also conclude that $c_2d_1=1$.  So let $c:= c_2$,  and observe that we have shown that $\tau_*(y_1)=cy_2$ and $\tau_*(y_2)= \smfrac{1}{c} y_1$ for some nonzero $c \in \Q$.

We now compute $\Bl_K(y_1,y_1)$,  $\Bl_K(y_1,y_2)$, $\Bl_K(y_2,y_1)$, and $\Bl_K(y_2,y_2)$,  relying on Proposition~\ref{prop:alexandblanch}.
Observe that
\[(A-tA^T)^{-1}
= \smfrac{-1}{\Delta_K(t)} \begin{pmatrix}
\ell(1-t) & mt- (m+1)\\ (m+1)t -m &0
  \end{pmatrix}
\]
where $\Delta_K(t)=(mt-(m+1))((m+1)t-m)$.
Therefore, $\Bl_K(b_1,b_1)= \smfrac{-\ell(1-t)^2}{\Delta_K(t)}$.
Using the fact that $\Bl_K(p(t)b_1, q(t)b_2)= p(t)q(t^{-1})\Bl_K(b_1,b_2)$ we therefore compute:
\begin{align*}
\Bl_K(y_1,y_1)&= ((m+1)t-m)((m+1)t^{-1}-m)\smfrac{-\ell(1-t)^2}{\Delta_K(t)}=0 \in \Q(t)/ \Q[t^{\pm1}]\\
\Bl_K(y_1,y_2)&=-\ell t^{-1} (1-t)^2 \smfrac{ (m+1)t-m}{mt-(m+1)}\\
\Bl_K(y_2,y_1)&=-\ell t^{-1} (1-t)^2 \smfrac{mt-(m+1)}{(m+1)t-m} \\
\Bl_K(y_2,y_2)&=0.
\end{align*}

Now,  let $v=(v_1,\dots,v_n)$ be any element of $\A(K_n)= \oplus^n \A(K)$.  Each $v_i= \lambda_i y_1^i + \mu_i y_2^i$ for some $\lambda_i, \mu_i \in \mathbb{Q}$.
We can then compute
\begin{align*}
\Bl_{K_n}(v, (\tau_n)_*(v))&= \sum_{i=1}^n \Bl_K(v_i, \tau_*(v_i)) \\
&= \sum_{i=1}^n \Bl_K\big( \lambda_i y_1+ \mu_i y_2 , \smfrac{\mu_i}{c} y_1 + \lambda_i c y_2 \big)\\
&=\sum_{i=1}^n \Big[ \Bl_K\big( \lambda_i y_1,  \lambda_i c y_2 \big)+ \Bl_K\big( \mu_i y_2,  \smfrac{\mu_i}{c} y_1 \big)\Big]\\
&=\sum_{i=1}^n \Big[ c (\lambda_i)^2 \Bl_K\big(y_1, y_2\big)+ \smfrac{(\mu_i)^2}{c}\Bl_K\big(y_2,  y_1\big)\Big]\\
&= -\ell t^{-1}(1-t)^2 \bigg[ \Big( c \sum_{i=1}^n (\lambda_i)^2 \Big) \smfrac{ (m+1)t-m}{mt-(m+1)}
+ \Big( \smfrac{\sum_{i=1}^n (\mu_i)^2}{c} \Big) \smfrac{mt-(m+1)}{(m+1)t-m}
\bigg].
\end{align*}
Applying Lemma~\ref{lemma:elem-fact-coprime}, this expression equals 0 in $\Q(t)/ \Q[t^{\pm1}]$ exactly when $c \sum_{i=1}^n (\lambda_i)^2=0= \frac{1}{c} \sum_{i=1}^n (\mu_i)^2$,  which occurs exactly when $\lambda_i=0=\mu_i$ for all $i=1, \dots, n$,  i.e.\ exactly when $v=0 \in \A(K_n)$.
\end{proof}

\section{Equivariant algebraic concordance}\label{section:equiv-alg-conc}

In this section we define an equivariant algebraic concordance group $\mathcal{AC}^{SI}$, we define a homomorphism $\Psi \colon \mathcal{C}^{SI} \to \mathcal{AC}^{SI}$, and we use the equivariant algebraic concordance group to show that the kernel of the forgetful map $F \colon \mathcal{C}^{SI} \to \mathcal{C}$ is infinite rank.

\subsection{An equivariant slice obstruction}\label{subsec-equiv-slice-obstruction}

We begin by proving the following obstruction to equivariant sliceness. This is presumably already known to experts,  but we could not find it in the literature.  Given a $\Qt$-module $H$ we write $|H|$ for its \emph{order}, which is an element of $\Qt$ well-defined up to multiplication by $\alpha t^k$, i.e. by units in $\Qt$.

We remind the reader that a strongly invertible knot $(K,  \tau)$ is \emph{equivariantly slice} if there exists a slice disc $D$ for $K$ and an extension of $\tau$ to a locally linear, order two homeomorphism $\wh{\tau}$ of $D^4$ such that $D =\wh{\tau}(D)$.
Unlike in our definition of concordance, we do not need to specify a direction on $K$.

\begin{proposition}\label{prop:equivslice}
Let $(K, \tau)$ be a strongly invertible knot.  If $(K, \tau)$ is equivariantly slice then there exists a submodule $P \leq \A^\Z(K)$ such that the following hold.
\begin{enumerate}
\item $P$ is a metabolizer for the integral Blanchfield pairing, i.e.\
\begin{enumerate}
\item for all $x,y \in P$ we have $\Bl^\Z(x,y)=0$;
\item $|P| \cdot \overline{|P|}= |\A(K)|$.
\end{enumerate}
\item $P$ is $\tau_*$-invariant, i.e.\ $\tau_*(P)=P$.
\end{enumerate}
The same holds with $\Lambda$ coefficients, for $\A(K)$ and the rational Blanchfield pairing $\Bl$.
\end{proposition}

\begin{proof}
Let $D$ be a slice disc for $K$,  and recall that $E_D:=D^4 \smallsetminus \nu(D)$ is a compact 4-manifold with $\partial E_D= M_K$.
 Let $P':= \ker(H_1(M_K; \Zt) \to H_1(E_D; \Zt))$, and let
 \[P := \{p \in H_1(M_K; \Zt) \mid kp \in P' \text{ for some } k \in \Z \sm \{0\}\}.\]
 It is well known \cite[Theorem~2.1]{Friedl-04}, \cite[Theorem~2.4]{Hillman-alg-invariants-links} that $P$ is a metabolizer for the Blanchfield pairing, establishing item (1).

 Now suppose that  $\tau$ extends over $D^4$ to $\widehat{\tau}$ with $D = \widehat{\tau}(D)$.
 It follows that
 \begin{align*}
 P'&= \ker(H_1(M_K; \Zt) \to H_1(E_D; \Zt))
  = \ker(H_1(M_K; \Zt) \to H_1(E_{\widehat{\tau}(D)}; \Zt))\\
 &= \tau_*(\ker(H_1(M_K; \Zt) \to H_1(E_D; \Zt)))=\tau_*(P').
 \end{align*}
Since $\tau_*$ is $\Z$-linear, $p \in P$ if and only if $kp \in P'$ for some $k \in \Z \sm \{0\}$, if and only if $k\tau_*(p) \in P'$ (because $k\tau_*(p) = \tau_*(kp) \in \tau_*(P') = P'$), if and only if $\tau_*(p) \in P$.
We have therefore established item (2), that $P$ is $\tau_*$-invariant.

The version with $\Lambda$ coefficients is easier; we can simply take $P:= \ker(H_1(M_K; \Lambda) \to H_1(E_D; \Lambda))$.
\end{proof}

This is an effective obstruction to equivariant sliceness.  For example,  when combined with Proposition~\ref{prop:genusone} it shows the following, which proves Corollary~\ref{cor:notslice} from the introduction. In many individual cases we expect this could also be proven using Sakuma's $\eta$ invariant, although it is not obvious how to apply that invariant to a general family of knots such as this.

\begin{corollary}~\label{cor:notslice}
Let $(K,\tau)$ be a genus one strongly invertible knot with nontrivial Alexander polynomial. Then $(K, \tau)$ is not equivariantly slice.
\end{corollary}
\begin{proof}
If $K$ is not algebraically slice then it is not even slice, so is certainly not equivariantly slice.
 Suppose that $K$ is algebraically slice with nontrivial Alexander polynomial.  If $(K,\tau)$ were equivariantly slice, there would be an invariant metabolizer $P$ for the Blanchfield form, by Proposition~\ref{prop:equivslice}.  Then for every $x \in P$ we would have $\Bl_{K}(x,\tau_*(x))=0$. But we computed in Proposition~\ref{prop:genusone} that this holds only for $x=0$. Since $\Delta_K \neq 1$, any such $P$ must be nontrivial by Proposition~\ref{prop:equivslice}~(1b).
 Thus there is no such~$P$.
\end{proof}

As noted in the introduction,  the proof of Proposition~\ref{prop:equivslice} carries through identically under the weaker hypothesis that $K$ bounds a slice disc $D$ such that for some extension $\widehat{\tau}$,  one has that $D$ and $\widehat{\tau}(D)$ are isotopic rel.\ boundary.  So Corollary~\ref{cor:notslice} also shows that genus one knots with nontrivial Alexander polynomials are not isotopy-equivariantly slice.

\subsection{The equivariant algebraic concordance group}\label{subsection:equiv-alg-conc}

Proposition~\ref{prop:equivslice} motivates the following definition, which we use to formalise the results on equivariant slicing.

Similarly to before,  given a $\Zt$-module $U$, we write $\ol{U}$ for the same abelian group as $U$ with the involuted $\Zt$ action, i.e.\ $(p , u) \mapsto \ol{p}\cdot u$.  Given a sesquilinear, Hermitian pairing $B \colon U \times U \to \Q(t)/\Zt$, there is an involuted pairing $\ol{B} \colon \ol{U} \times \ol{U} \to \Q(t)/\Zt$ given by $\ol{B}(x,y) = \ol{B(x,y)}$. This is also Hermitian but has the opposite convention on the meaning of sesquilinearity, that is $\ol{B}(px,qy) = \ol{p} \ol{B}(x,y) q$.

\begin{definition}\label{defn:equiv-alg-conc}
We introduce a set and an equivalence relation which will lead to a definition of the equivariant algebraic concordance group.
\begin{enumerate}
\item   We consider the set of triples $(H,\Bl, \tau)$, consisting of the following data.
  \begin{enumerate}[(i)]
\item A finitely generated $\Zt$-module $H$, that is $\Zt$-torsion, $\Z$-torsion free, and such that $m_{1-t} \colon H \to H$; $x \mapsto (1-t)\cdot x$  is an isomorphism.
\item A sesquilinear, Hermitian, nonsingular pairing $\Bl \colon H \times H \to \Q(t)/\Zt$.
\item An anti-isometry $\tau \colon H \to H$ with $\tau^2 = \Id$. That is, $\tau \colon H \to H$ is an anti-automorphism, or in other words a $\Zt$-module isomorphism $\tau \colon H \xrightarrow{\cong} \ol{H}$, that induces an isometry between $\Bl \colon H \times H \to \Q(t)/\Zt$ and $\ol{\Bl} \colon \ol{H} \times \ol{H} \to \Q(t)/\Zt$.
\end{enumerate}
We call a triple $(H,\Bl, \tau)$ an \emph{abstract equivariant Blanchfield pairing}.
\item
An \emph{isometry} of  abstract equivariant Blanchfield pairings $\theta \colon (H_1,\Bl_1, \tau_1) \xrightarrow{\cong} (H_2,\Bl_2, \tau_2)$ is an isometry $\theta \colon H_1 \to H_2$ of Blanchfield pairings such that $\theta \circ \tau_1 = \tau_2 \circ \theta$.
\item
We say that $(H,\Bl, \tau)$ is \emph{metabolic} if there is a $\Zt$ submodule $P \subseteq H$, called a \emph{metabolizer}, such that
\begin{enumerate}[(a)]
\item for all $x,y \in P$ we have $\Bl(x,y)=0$;
\item $|P| \cdot \overline{|P|}= |H|$;
\item $P$ is $\tau$-invariant, i.e.\ $\tau(P)=P$.
\end{enumerate}
\item The sum of two abstract equivariant Blanchfield pairings $(H_1,\Bl_1, \tau_1)$ and $(H_2,\Bl_2, \tau_2)$ is
\[(H_1,\Bl_1, \tau_1) \oplus (H_2,\Bl_2, \tau_2) := (H_1 \oplus H_2, \Bl_1 \oplus \Bl_2, \tau_1 \oplus \tau_2).\]
\item
We say that two abstract equivariant Blanchfield pairings $(H_1,\Bl_1, \tau_1)$ and $(H_2,\Bl_2, \tau_2)$ are \emph{algebraically concordant} if there are metabolic pairings $(U_1,B_1,\sigma_1)$ and $(U_2,B_2,\sigma_2)$ such that there is an isometry
\[(H_1,\Bl_1, \tau_1) \oplus (U_1,B_1,\sigma_1) \cong   (H_2,\Bl_2, \tau_2) \oplus (U_2,B_2,\sigma_2). \]
It is easy to see that algebraic concordance is an equivalence relation.
\end{enumerate}
\end{definition}

\begin{remark}
  Does stably metabolic imply metabolic? If so, we could simplify the equivalence relation to requiring that $(H_1,\Bl_1, \tau_1) \oplus (H_2,-\Bl_2, \tau_2)$ is metabolic.
\end{remark}

\begin{proposition}\label{prop:its-a-group}
  With respect to the given addition, the set of algebraic concordance classes of abstract equivariant Blanchfield pairings forms a group. The inverse of $(H,\Bl, \tau)$ is $(H,-\Bl, \tau)$.
\end{proposition}

We call this group the \emph{equivariant algebraic concordance group}, and denote it $\mathcal{AC}^{SI}$.
If $(H,\Bl, \tau) =0 \in \mathcal{AC}^{SI}$ then we say that $(H,\Bl, \tau)$ is \emph{equivariantly algebraically slice}. Similarly, if a strongly invertible knot $(K,\tau)$ lies in $\ker \Psi$, then we say that $(K,\tau)$ is \emph{equivariantly algebraically slice}.

\begin{proof}[Proof of Proposition~\ref{prop:its-a-group}]
It is straightforward to argue that the addition is well-defined on equivalence classes, that it is associative, and that the equivalence class containing all metabolic abstract equivariant Blanchfield pairings is the identity.
We need to prove that the inverse of $(H,\Bl, \tau)$ is $(H,-\Bl, \tau)$, or in other words that there is a metabolic pairing
$(U,B,\sigma)$ such that
\[(H,\Bl, \tau) \oplus (H,-\Bl, \tau) \oplus (U,B,\sigma)\]
is metabolic.  In fact we can take $U=0$, and define the diagonal submodule
\[P:= \{(x,x) \in H \oplus H \mid x \in H\}.\]
We check that $P$ is a metabolizer for $\Bl \oplus - \Bl$.
To see (a), we compute that $(\Bl \oplus -\Bl) ((x,x),(y,y)) = \Bl(x,y) - \Bl(x,y) = 0$ for all $x,y \in H$ and therefore for every $(x,x)$ and $(y,y)$ in~$P$. To show (b), note that since $\Bl$ is nonsingular we know that $|H| = \ol{|H|}$, and therefore $|H \oplus H| = |H| \cdot |H| = |H|\cdot \ol{|H|}$. On the other hand $|P|= |H|$, and so $|P| \cdot \ol{|P|} = |H|\cdot \overline{|H|} = |H \oplus H|$.  Finally, to see (c) we compute that for any $(x,x) \in P$ we have $(\tau\oplus \tau)(x,x) = (\tau(x),\tau(x)) \in P$. Therefore $P$ is $\tau$-invariant.
This completes the proof that $P$ is a metabolizer, and therefore completes the proof that $\mathcal{AC}^{SI}$ is a group.
\end{proof}

\begin{proposition}\label{prop:its-a-homomorphism}
  Taking the integral Blanchfield form of a strongly invertible knot $(K,\tau)$ together with the involution-induced map on the integral Alexander module $\A^\Z(K)$ gives rise to a  homomorphism $\Psi \colon \mathcal{C}^{SI} \to \mathcal{AC}^{SI}$.
\end{proposition}

\begin{proof}
We showed in Proposition~\ref{prop:anti-isometry} that the involution-induced map $\tau_* \colon \A^{\Z}(K) \to \A^{\Z}(K)$ is an anti-isometry of the Blanchfield pairing. Thus we obtain an element of the codomain $\mathcal{AC}^{SI}$.
We know from the ordinary algebraic concordance group that $\Psi(-K,\tau) = (\A^\Z(K),-\Bl_K^\Z,\tau_*) = - \Psi(K,\tau)$.

We check that $\Psi$ is well-defined. The argument is at this stage standard and purely formal.
Suppose that $(K_1,\tau_1)$ and $(K_2,\tau_2)$ are equivariantly concordant. Then $(K_1 \# -K_2,\tau_1 \cup \tau_2)$ is equivariantly slice, and therefore
\begin{align*}
  (\A^\Z(K_1 \# K_1), \Bl_{K_1 \# -K_2}^{\Z}, \tau_1 \cup \tau_2) &\cong (\A^\Z(K_1), \Bl_{K_1}^{\Z},(\tau_1)_*) \oplus  (\A^\Z(K_2), -\Bl_{K_2}^{\Z},(\tau_2)_*) \\
  &= \Psi(K_1,\tau_1) \oplus -\Psi(K_2,\tau_2)
\end{align*}
is a metabolic form $(U,B,\sigma)$ by Proposition~\ref{prop:equivslice}.
We also used Proposition~\ref{prop:connected-sum-well-behaved} here.
Add $\Psi(K_2,\tau_2)$ to both sides to see that
\begin{align*}
  \Psi(K_1,\tau_1) \oplus  -\Psi(K_2,\tau_2) \oplus \Psi(K_2,\tau_2) \cong (U,B,\sigma) \oplus \Psi(K_2,\tau_2).
\end{align*}
On the left hand side, $-\Psi(K_2,\tau_2) \oplus \Psi(K_2,\tau_2)$ is metabolic, as we showed in the proof of Proposition~\ref{prop:its-a-group}.  Since $(U,B,\sigma)$ is also metabolic, it follows that $\Psi(K_1,\tau_1) = (\A^\Z(K_1), \Bl_{K_1}^{\Z},(\tau_1)_*)$ and $\Psi(K_2,\tau_2) = (\A^\Z(K_2), \Bl_{K_2}^{\Z},(\tau_2)_*)$ are algebraically concordant. Thus $\Psi \colon \mathcal{C}^{SI} \to \mathcal{AC}^{SI}$ is a well-defined map as desired.

Finally, we know by Proposition~\ref{prop:connected-sum-well-behaved} and the observation in the first paragraph of the proof that for every pair of strongly invertible knots $(K_1,\tau_1)$ and $(K_2,\tau_2)$, we have that \[\Psi((K_1,\tau_1) \# -(K_2,\tau_2)) = \Psi(K_1,\tau_1) \oplus \Psi(-K_2,\tau_2) = \Psi(K_1,\tau_1) \oplus -\Psi(K_2,\tau_2).\]
It follows that $\Psi$ is indeed a homomorphism.
\end{proof}

\subsection{The kernel of $F$}

We consider the forgetful map $F \colon \mathcal{C}^{SI} \to \mathcal{C}$.   Recall that Theorem~\ref{thm:subgroup-of-ker-F} asserts that $\ker F$ contains a subgroup of infinite rank, which is detected by the equivariant algebraic concordance group.
Combining Propositions ~\ref{prop:genusone} and ~\ref{prop:equivslice}   implies the following.

\begin{theorem}\label{theorem:on-ker-F--in-body}
Let $K_1,\dots, K_n$ be genus one  algebraically slice knots with nontrivial and pairwise distinct Alexander polynomials and strong inversions $\tau_i$.
Let $a_1, \dots, a_n \in \N$,  and let $(a_iK_i,  a_i\tau_i)$ denote the $a_i$-fold equivariant connected sum of $(K_i, \tau_i)$
The knot $(K, \tau)=\#_{i=1}^n (a_i K_i,  a_i\tau_i)$ is not equivariantly algebraically slice and is therefore not equivariantly slice.
\end{theorem}

\begin{proof}
We work with $\Lambda$ coefficients.
Let $x=(x_1, \dots, x_n)$ be an element of $\A(K)= \oplus_{i=1}^n \A(a_i K_i)$.
Since the Alexander polynomials of $K_1, \dots, K_n$ are distinct degree 2 symmetric polynomials satisfying $|p(1)|=1$,  they are pairwise relatively prime.  By the multiplicativity of Alexander polynomials under connected sum,
the Alexander polynomials of $a_1K_1,\dots, a_nK_n$ are also pairwise relatively prime.   It follows that
\[ \Bl_K(x, \tau_*(x))= \sum_{i=1}^n \Bl_{a_i K_i}(x_i, (a_i \tau_i)_*(x_i))=0\]
if and only if $Bl_{a_i K_i}(x, (a_i \tau_i)_*(x_i))=0$ for all $i=1, \dots, n$.
By Proposition~\ref{prop:genusone}, for each $i=1, \dots, n$ we have that $\Bl_{a_i K_i}(x_i, (a_i \tau_i)_*(x_i))= 0$ if and only if $x_i=0$.  Therefore $\Bl_K(x, \tau_*(x))=0$ if and only if $x=0$,  and there certainly is no $\tau_*$ invariant metabolizer for the Blanchfield pairing of $K$.   Therefore $(K,\tau)$ is not equivariantly algebraically slice and by Proposition~\ref{prop:equivslice},  $K$ is not equivariantly slice.
\end{proof}

It is now straightforward to prove Theorem~\ref{thm:subgroup-of-ker-F} from the introduction, which follows from the next corollary.

\begin{corollary}
  Let $\{K_i\}_{i=1}^{\infty}$ be a collection of strongly invertible genus one slice knots with nontrivial and pairwise distinct Alexander polynomials. Then the $\{K_i\}_{i=1}^{\infty}$ generate an infinite rank subgroup of $\ker (F \colon \mathcal{C}^{SI} \to \mathcal{C})$ whose image in $\mathcal{AC}^{SI}$ is also infinite rank.
\end{corollary}

\begin{proof}
It suffices to check that for every linear combination of the $K_i$, $J:= \#_i b_i K_i$, with $b_i \neq 0$ for finitely many $i$, we have $\Psi(J) \neq 0$, and therefore $J$ is not equivariantly slice.  If $b_i \geq 0$ then set $a_i := b_i$ and write $K_i' := K_i$, while if $b_i <0$ then set $a_i = -b_i$ and $K_i' := -K_i$. Then note that $J = \#_i a_i K_i'$.  The $K_i'$ have nontrivial pairwise distinct Alexander polynomials, are genus one, and are strongly invertible. Therefore Theorem~\ref{theorem:on-ker-F--in-body} applies to show that $J$ is not equivariantly algebraically slice, and so $J$ is not equivariantly slice.
 This shows that $\{\Psi(K_i)\}_{i=1}^{\infty}$ is an infinite rank subgroup of  $\mathcal{AC}^{SI}$, and therefore that the $\{K_i\}_{i=1}^{\infty}$ generate an infinite rank subgroup of $\ker F$ as claimed.
\end{proof}

\subsection{Some amphichiral examples}

In this subsection we show that many order two knots in $\mathcal{C}$ map to infinite order equivariant Blanchfield pairings in $\mathcal{AC}^{SI}$, and so are also infinite order in $\mathcal{C}^{SI}$.

Let $K:=K_a$ be a generalized twist knot with continued fraction expansion $[2a,2a]^+$ for some $a>0$,  with axis of strong inversion $\gamma$ as indicated in Figure~\ref{fig:amphi}.
\begin{figure}[h!]
\includegraphics[height=3cm]{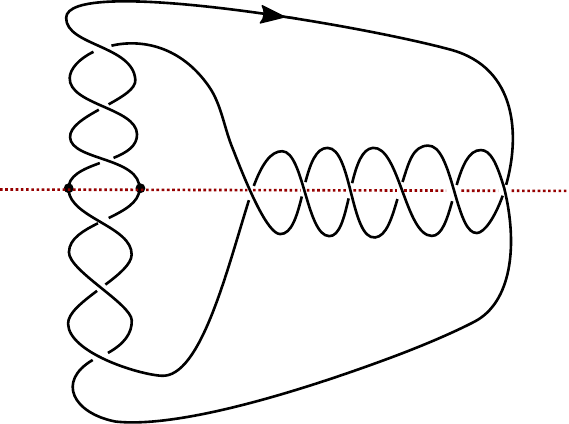}
\caption{The knot $K_a$ for $a=3$ with an axis of strong inversion.}
\label{fig:amphi}
\end{figure}

Applying Seifert's algorithm to the diagram of Figure~\ref{fig:amphi},  we see that $K$ has a genus one Seifert surface $F$ and basis $\alpha_1,\alpha_2$ for $H_1(F)$ with corresponding Seifert matrix $A=\bsm a & 0 \\ 1 & -a \esm$.
Let $\beta_1, \beta_2$ be the corresponding dual generating set for $H_1(S^3 \smallsetminus F)$,  and observe that $\tau(\beta_2)= \beta_2$.

Following the notation of Proposition~\ref{prop:alexandblanch},  we have that $\A(K)$ is generated by $b_1, b_2$ and has relations $a(t-1) b_1+t b_2=0$ and  $-b_1 -a(t-1)b_2=0$.  Simplifying this gives that $\A(K) \cong \Q[t^{\pm1}]/ \langle p_a(t) \rangle$,  generated by $b_2$, where $p_a(t)= \Delta_{K}(t)= a^2 t^2- (2a^2-1)t+a^2$.
Additionally,  we have that $\tau_*( q(t) b_2)= q(t^{-1}) b_2$ for all $q(t) \in \Q[t^{\pm1}]$.

\begin{lemma}\label{lem:irreducible-poly}
For every $a > 0$,  $p_a(t)$ is irreducible over $\Qt$,  and hence $\A(K)\cong \Qt/ p_a(t)$ is a field.
\end{lemma}

\begin{proof}
  Since $p_a(t)$ is a degree 2 symmetric polynomial with integer coefficients that evaluates to $1$ at $t=1$,  it suffices to show that $p_a(t)$ cannot be written  as $t^{2} q(t)q(t^{-1})$ for any $q(t) \in \Zt$.  But this follows from the fact that $p_a(-1)=4a^2-1= (2n)^2-1$ is never a square.
\end{proof}

\begin{proposition}\label{prop:amphi}
For every $n \in \mathbb{N}$, the knot $\#^n(K_a,  \tau)$ is not equivariantly slice.
\end{proposition}

For $a=1,2$ we already know using Sakuma's computations of the $\eta$ invariant~\cite{Sakuma} that the knots $\#^n(K_a,  \tau)$ are not equivariantly slice.

\begin{proof}
First,  note that when $n$ is odd,  $\#^n K_a$ is concordant to $K_a$,  which is not slice. So we can assume that $n=2m$ is even.

Now let $H \leq \A(\#^n K_a)$ be a $(\tau_*)$-invariant submodule of order $p_a(t)^m$; that is,  a potential $(\tau_*)$-invariant metabolizer for the Blanchfield pairing. Since
\[ \A(\#^n K_a)\cong \bigoplus^n \A(K_a)= \bigoplus^n \Q[t^{\pm1}]/  p_a(t),\]
we know that $H \cong \left(  \Qt/  p_a(t)  \right)^m$.

Furthermore,  after rearranging our summands if necessary, the submodule $H \leq \A(K_a)$ is generated as a $\Qt/ p_a(t)$-module by
\begin{align*}
x_1 &=(1,0,\dots,0, q_1^1, \dots, q_1^m)\\
x_2 &= (0,1,0, \dots, 0, q_2^1, \dots, q_2^m)\\
&\phantom{n} \vdots   \\
x_m &= (0, \dots, 0, 1 , q_m^1, \dots, q_m^m)
\end{align*}
for some $q_i^j \in \Q[t^{\pm1}]/ \langle p_a(t) \rangle$, $1 \leq i, j \leq m$.
This  relies on the fact that since $p_a(t)$ is irreducible by Lemma~\ref{lem:irreducible-poly},   $\Qt/ p_a(t)$ is a field and so $\A(K) \cong (\Qt/ p_a(t))^n$ is a vector space. Therefore the existence of a generating set of this form follows from some elementary linear algebra.

Write each $q_1^j= c_j + d_j t$ for some $c_j, d_j \in \Q$.
Observe that
\begin{align*}
(\#^n \tau)_*(x_1)&= (\tau_*(1), \tau_*(0),\dots,\tau_*(0), \tau_*(q_1^1), \dots, \tau_*(0)(q_1^m))\\
&= (\tau_*(1), \tau_*(0),\dots,\tau_*(0), \tau_*(c_1+d_1t), \dots, \tau_*(0)(c_m + d_mt))\\
&= (1, 0, \dots, 0, c_1 + d_1 t^{-1},  \dots, c_m + d_m t^{-1}).
\end{align*}
Since $(\#^n \tau)_*(x_1) \in H$,  we can write $(\#^n \tau)_*(x_1)= \sum_{i=1}^n r_i(t) x_i$ for some $r_i(t) \in \Q[t^{\pm1}]/ \langle p_a(t) \rangle$.  Considering our expressions for $x_1, \dots, x_m$ and for $(\#^n \tau)_*(x_1)$ and looking at the first $m$ coordinates,  we obtain that$r_1(t)=1$ and $r_2(t)=\cdots=r_m(t)=0$.   Looking at the last $m$ coordinates, we can conclude that $d_1=\cdots=d_m=0$.

So $x_1= (1, 0, \dots, 0, c_1, c_2, \dots, c_m)$ for some $c_i \in \Q$.
But we can now show that $\Bl_{\#^n K_a}(x_1,x_1) \neq 0$,  and hence that $H$ is not a metabolizer:
\begin{align*}
 \Bl_{\#^n K_a}(x_1,x_1) = \Bl_{K_a}(1,1)+ \sum_{i=1}^m \Bl_{K_a}(c_i, c_i)= \Big(1+ \sum_{i=1}^m c_i^2 \Big) \Bl_{K_a}(1,1).
\end{align*}
Since $\A(K) \cong \Q[t^{\pm1}]/ \langle p_a(t) \rangle$ and $\Bl_{K_a}$ is nondegenerate,  we have that $\Bl_{K_a}(1,1)$ is nonzero.  (Of course,  we could also directly compute this using the Seifert matrix and Proposition~\ref{prop:alexandblanch}.)
Therefore $ \Bl_{\#^n K_a}(x_1,x_1)$ is nonzero as well.
\end{proof}

\section{A lower bound on the equivariant 4-genus}\label{section:lower-bound}

Now we switch our attention to proving the lower bound from Theorem~\ref{thm-main-obstruction-intro}, which will lead to the proof of Theorem~\ref{thm:intro-arb-large} when combined with the computation in Proposition~\ref{prop:genusone}.

\subsection{Construction of the 4-manifold \texorpdfstring{$Z$}{Z} and its properties}\label{section:construction-of-Z}

As in the proof of Proposition~\ref{prop:equivslice},  we extensively consider the kernel of the inclusion induced map $H_1(E_K; \Lambda) \to H_1(E_F; \Lambda)$,  where $F$ is some locally flat surface in $D^4$ with boundary $K$.  However,  it will simplify our arguments to work with the closed 3-manifold  $M_K$  and an associated 4-manifold $Z$ with $\partial Z= M_K$ instead.

\begin{proposition}\label{prop:factsaboutZ}
Let $(K,\tau)$ be a strongly invertible knot in $S^3$ bounding a genus $g$ surface $F$ in $D^4$.
There exists a 4-manifold $Z$ with boundary $M_K$ such that the following hold.
\begin{enumerate}
\item \label{item:h1z} The inclusion-induced map $i_*\colon H_1(M_K; \Z) \to H_1(Z; \Z)$  is an isomorphism.
\item \label{item:relh1lambda} $H_1(Z,  M_K; \Lambda)$ and $H_1(Z; \Lambda)$ are torsion.
\item \label{item:h2zrank} The free part of $H_2(Z; \Lambda)$ has rank $2g$.
\item \label{item:kernelssame} The inclusion-induced map $i_* \colon H_1(E_K; \Lambda) \to H_1(M_K; \Lambda)$ is an isomorphism under which $\ker( H_1(E_K; \Lambda) \to H_1(E_F; \Lambda) )$ is mapped to $\ker( H_1(M_K; \Lambda) \to H_1(Z; \Lambda))$.
\item \label{item:kernelinvariant} If $\tau$ extends to an involution $\widehat{\tau}$ on $D^4$ such that $F=\widehat{\tau}(F)$,  then $H:=\ker( H_1(M_K; \Lambda) \to H_1(Z; \Lambda))$ is invariant, i.e.\ $\tau_*(H)=H$.
\end{enumerate}
\end{proposition}

For the proof of Proposition~\ref{prop:factsaboutZ} we will need the following special case of  \cite[Propositions~2.9~and~2.11]{COT1}.

\begin{proposition}\label{prop:h0h1cot}
Let $X$ be a $($space with the homotopy type of a$)$ finite CW complex, and let $\phi \colon \pi_1(X) \to \Z$ be a nontrivial representation.
Then $H_0(X; \Q(t))=0$
and $\dim_{\Q(t)} H_1(X; \Q(t)) \leq b_1(X)-1$.
\end{proposition}

\begin{proof}[Proof of Proposition~\ref{prop:factsaboutZ}]
Define
\[ Z:= (D^4 \smallsetminus \nu(F)) \bigcup_{S^1 \times F} (S^1 \times H), \]
 where $H$ is a genus $g$ handlebody with boundary $\partial H= F \cup D^2$.  To make this gluing we choose a framing of the normal bundle of $F$ in $D^4$ such that for each simple closed curve $\alpha \subseteq F$, the curve $\alpha \times \{1\} \subseteq S^1 \times F \subseteq D^4\smallsetminus \nu F$ is null-homologous.  There is also a choice of precisely which handlebody $H$ we choose to fill $F \cup D^2$.   We make an arbitrary choice here; this does not affect the homological properties of $Z$ that we will use.  A Mayer-Vietoris argument establishes item~\eqref{item:h1z},  as well as the fact that $H_2(Z) \cong \Z^{2g}$.   We note for later use that item~\eqref{item:h1z} also implies that $H_3(Z; \Q) \cong H^3(Z; \Q) \cong H_1(Z, M_K; \Q) \cong 0$,  and so $\chi(Z)= 1-1+2g+0+0=2g$.

In order to establish items  \eqref{item:relh1lambda} and \eqref{item:h2zrank},  by the flatness of $\Q(t)$ as a $\Lambda$-module it suffices to show that $H_1(Z, M_K; \Q(t))=0$ and $H_2(Z; \Q(t)) \cong \Q(t)^{2g}$.   By Proposition~\ref{prop:h0h1cot},  we have that $H_i(M_K, \Q(t))=0=H_i(Z; \Q(t))$ for $i=0,1$.  It follows from the long exact sequence of $(Z, M_K)$ with $\Q(t)$ coefficients that $H_1(Z, M_K; \Q(t))=0$ as desired for item \eqref{item:relh1lambda}.
Item \eqref{item:h2zrank} now quickly follows from an Euler characteristic computation for $Z$ using $\Q(t)$ coefficients.  We  already know $H_0(Z; \Q(t))=0=H_1(Z, \Q(t))$.  Additionally,  for $i=3,4$ we have that $H_i(Z; \Q(t))= H^i(Z; \Q(t))= H_{4-i}(Z, M_K; \Q(t))=0$,  where the first equality comes from universal coefficients and the second from Poincar{\'e} duality.  We are now ready to recover item \eqref{item:h2zrank},  since
\begin{align*}
 2g&= b_0^{\Q(t)}(Z)- b_1^{\Q(t)}(Z)+b_2^{\Q(t)}(Z)-b_3^{\Q(t)}(Z)+b_4^{\Q(t)}(Z)\\
 &= 0-0+\dim H_2(Z; \Q(t))-0+0.
 \end{align*}

We now wish to establish item \eqref{item:kernelssame}.  Recall that
$M_K= E_K \cup_{T^2} (S^1 \times D^2)$.  Since $H_1(T^2; \Lambda)= H_1(S^1 \times D^2; \Lambda)=0$ and $H_0(T^2; \Lambda) \to H_0(S^1 \times D^2; \Lambda)$ is an isomorphism,  the Mayer-Vietoris sequence for $M_K$ immediately implies that $i_* \colon H_1(E_K; \Lambda) \to H_1(M_K; \Lambda)$ is an isomorphism.
Now recall that $Z= E_F \cup_{S^1 \times F} (S^1 \times H)$,  where $H$ is a genus $g$ handlebody with $\partial H= F \cup D^2$.  This decomposition is compatible with that of $M_K$,  and so we obtain the following commutative diagram,  where all maps are induced by inclusion:
\begin{center}
\begin{tikzcd}\label{bigbox}
H_1(E_K; \Lambda) \arrow[r, "i_*"] \arrow[d, "f_*"]
&  H_1(M_K; \Lambda)= H_1(E_K \cup (S^1 \times D^2); \Lambda) \arrow[d, "g_*"] \\
H_1(E_{F}; \Lambda) \arrow[r, "j_*"] & H_1(Z; \Lambda)= H_1(E_{F} \cup (S^1 \times H); \Lambda).
\end{tikzcd}
\end{center}
We wish to show that $\ker(g_*)= i_*(\ker(f_*))$.   One containment is immediate: for $y \in i_*(\ker(f_*))$ write $y= i_*(x)$ for $x \in \ker(f_*)$ and observe that  $y \in \ker(g_*)$,  since $g_*(y)= g_*i_*(x)=j_*f_*(x)=j_*(0)=0$.

Now let $y \in \ker(g_*)$ and,  recalling that $i_*$ is a isomorphism,  let $x \in H_1(E_K; \Lambda)$ be such that $i_*(x)=y$ in order to show that $f_*(x)=0$.
Since $j_*f_*(x)= g_*i_*(x)=g_*(y)=0$,  we certainly have that $f_*(x)$ is in $\ker(j_*)$.  Now consider the following portion of the Mayer-Vietoris sequence for $Z$:
\[ H_1(S^1 \times F; \Lambda) \to H_1(E_{F}; \Lambda) \oplus H_1(S^1 \times H; \Lambda) \to H_1(Z; \Lambda).\]
Since this sequence is exact and $f_*(x)$ maps to $0$ in $H_1(Z; \Lambda)$,  we can conclude that $f_*(x) \in \im(H_1(S^1 \times F; \Lambda) \to H_1(E_{F}; \Lambda))$.
One can compute directly that
\[H_1(S^1 \times F; \Lambda)= H_1(\R \times F; \Q) \cong \left(\lambda/ \langle t-1 \rangle\right)^{2g(F)},\]
and hence that $f_*(x)$ is annihilated by $t-1$.
But since the order of $H_1(E_K; \Lambda)$ is $\Delta_K(t)$,  it is also true that $f_*(x)$ is annihilated by $\Delta_K(t)$.
Note that $\Delta_K(t)$ and $t-1$ are relatively prime. (To see this, it is straightforward to find a $p \in \Zt$ such that $\Delta_K(t) + p(t-1) = m \in \Z$. Then use that $\Delta_K(1) = \pm 1$.) Therefore since $f_*(x)$ is annihilated by relatively prime polynomials, we have as desired that $f_*(x)=0$. This completes the proof of item \eqref{item:kernelssame}.

To prove item \eqref{item:kernelinvariant},  suppose that $\tau$ extends to an involution $\widehat{\tau}$ on $D^4$ such that $F=\widehat{\tau}(F)$.
It follows immediately that
\begin{align*}
\ker(H_1(E_K; \Lambda) \to H_1(E_F; \Lambda))= \ker(H_1(E_K; \Lambda) \to H_1(E_{\widehat{\tau}(F)}; \Lambda)).
\end{align*}
Note that for any surface $G$ in $D^4$ with $\partial G= K$ we have
\[  \ker(H_1(E_K; \Lambda) \to H_1(E_{\widehat{\tau}(G)}; \Lambda))= \tau_*(\ker(H_1(E_K; \Lambda) \to H_1(E_{G}; \Lambda))).\]
Therefore,  $\ker(H_1(E_K; \Lambda) \to H_1(E_F; \Lambda))$ is $\tau_*$-invariant.  However,  by item $(4)$ we know that $\ker(H_1(E_K; \Lambda) \to H_1(E_F; \Lambda))$ is identified with $\ker(H_1(M_K; \Lambda) \to H_1(Z; \Lambda))$ via the inclusion induced map,  which is compatible with $\tau_*$,  and so we get our desired result.
\end{proof}

\subsection{Blanchfield forms and generating rank}

The proof of the next proposition is closely related to a standard argument, but since we need a slight variation we give the details.

\begin{proposition}\label{prop:vanishingofbl}
Let $K$ be a knot in $S^3$ with zero surgery $M_K$,  and suppose $Z$ is a 4-manifold with $\partial Z= M_K$ such that $i_* \colon H_1(M_K) \to H_1(Z)$ is an isomorphism.  Suppose that $H_1(Z; \Lambda)$ is $\Lambda$-torsion.
Then for every $x \in TH_2(Z, M_K; \Lambda)$ and every $y \in \ker(i_* \colon H_1(M_K; \Lambda) \to H_1(Z; \Lambda))$ we have
$\Bl(\partial x,y)=0.$
\end{proposition}

\begin{proof}
Consider the following diagram,  recalling that $H_1(M_K; \Lambda)$ and $H_1(Z; \Lambda)$ are both $\Lambda$-torsion.
\begin{center}
\begin{tikzcd}\label{bigbox}
TH_2(Z, M_K; \Lambda) \arrow[r, "\partial |_T"] \arrow[d, "PD^{-1}"] \arrow[ddd,  bend right=60, "\beta"']
&  H_1(M_K;  \Lambda) \arrow[r, "i_*"] \arrow[d, "PD^{-1}"'] \arrow[ddd, bend left=60, "\Bl"]
& H_1(Z; \Lambda) \\
TH^2(Z; \Lambda) \arrow[d, "B^{-1}"] \arrow[r,"i^*"] & H^2(M_K; \Lambda) \arrow[d, "B^{-1}"'] & \\
 H^1(Z; Q/\Lambda) \arrow[d,  "\kappa"] \arrow[r,"i^*"] & H^1(M_K; Q/\Lambda) \arrow[d,  "\kappa"'] & \\
\Hom(H_1(Z; \Lambda),  Q/ \Lambda) \arrow[r, "i_*^{\wedge}"] & \Hom(H_1(M; \Lambda), Q/ \Lambda). &
\end{tikzcd}\\
\end{center}
While the top row is not necessarily exact, we do have that $\im(\partial |_T) \subseteq \ker(i_*)$,  since
\[H_2(Z, M_K; \Lambda) \xrightarrow{\partial} H_1(M_K; \Lambda) \xrightarrow{i_*} H_1(Z; \Lambda)\]
is exact.  Moreover,  since all of the vertical maps are natural the diagram commutes.  This is straightforward for the Bockstein and universal coefficients, while \cite[Theorem~IV.9.2]{Br93} shows that the top square commutes.

Now let $x \in TH_2(Z, M_K; \Lambda)$ and $y \in \ker(i_* \colon H_1(M_K; \Lambda) \to H_1(Z; \Lambda)$.  We therefore have that
\[ \Bl(\partial x)(y)= i_*^{\wedge}(\beta(x))(y)= \beta(x)(i_*(y))= \beta(x)(0)=0.\]
The first equality comes from the commutativity of the diagram, the second equality from the definitional relationship between $i_*^{\wedge}$ and $i_*$,  and the last from our assumption on $y$.
\end{proof}

In order to effectively apply Proposition~\ref{prop:factsaboutZ},  we will need to show that $\partial(TH_2(Z, M_K; \Lambda))$ has large generating rank.  It will be useful to have the following facts about the generating rank of finitely generated modules over PIDs,  which follow from the fundamental theorem of finitely generated modules over PIDs; see also \cite[Lemma~4.1]{Miller-Powell-19}.

\begin{proposition} \label{prop:genrank}
Let $A, B$ be finitely generated modules over a PID $S$.
\begin{enumerate}
\item\label{item:subset} If $A \subseteq B$ then $\grk A \leq \grk B$.
\item \label{item:map} If $f \colon A \to B$ is a map of $S$-modules, then
\[ \grk \im(f) \leq \grk A \leq \grk \im(f)+ \grk \ker(f).\]
\end{enumerate}
\end{proposition}

The next proposition is one of the key technical facts on generating ranks.

\begin{proposition}\label{prop:nontrivialitypartial}
Let $Z$ be a compact, oriented 4-manifold with boundary $\partial Z=M_K$ such that $i_* \colon H_1(M_K;\Z) \to H_1(Z;\Z)$ is an isomorphism.
Let $n$ be the $\Lambda$-rank of $H_2(Z;\Lambda)$ i.e.\ the free part of $H_2(Z; \Lambda)$ is isomorphic to $\Lambda^n$.
Assume that $H_1(Z, M_K; \Lambda)$ is torsion.
Then the generating rank of $\partial \left( TH_2(Z, M_K; \Lambda) \right)$ is at least $\frac{1}{2}\grk \A(K)- n.$
\end{proposition}

For the proof of Proposition~\ref{prop:nontrivialitypartial},  we will  need the following result from our article~\cite{Cha-Miller-Powell} with Jae Choon Cha.

\begin{lemma}[{\cite[Lemma~7.5]{Cha-Miller-Powell}}]\label{lem:kercoker}
Let $X$ be a compact, oriented 4-manifold with boundary $\partial X=Y$. Let $S$ be a commutative PID with no zero-divisors,
and suppose there is a representation $\Phi$ of the fundamental group of $Y$ into $\Aut(S)$ that extends over $X$.
Consider the long exact sequence of the pair $(X,Y)$: 
\[ \cdots \to H_2(X;S) \xrightarrow{j_2} H_2(X,Y;S) \xrightarrow{\partial} H_1(Y;S) \xrightarrow{i_1} H_1(X;S) \xrightarrow{j_1} H_1(X,Y;S) \to \cdots\]
If $H_1(X,Y;S)$ is torsion, then $\ker(j_1|_T)$ and $\coker(j_2|_T)$ are isomorphic as $S$-modules.
\end{lemma}

\begin{proof}[Proof of Proposition~\ref{prop:nontrivialitypartial}]
Choose an isomorphism $H_2(Z; \Lambda)\cong \Lambda^n \oplus TH_2(Z; \Lambda)$ and (noting that the free part of $H_2(Z, M_K; \Lambda)$ must also have rank $n$ by duality and universal coefficients) choose an isomorphism $H_2(Z, M_K; \Lambda) \cong \Lambda^n \oplus TH_2(Z,M_K;\Lambda)$.  This allows us to decompose the long exact sequence of the pair $(Z, M_K)$ as follows,  where all homology is taken with coefficients in $\Lambda$:
\[
\begin{tikzcd}
\Lambda^n \arrow[r, "j_2^a"] \arrow[ddr, "j_2^b"] & \Lambda^n \arrow[dr, "\partial^a"] &&&& \\
 \oplus & \oplus & H_1(M_K) \arrow[r,"i_1"] & H_1(Z)  \arrow[r, "j_1"]& H_1(Z, M_K)  \arrow[r] & 0 \\
 TH_2(Z) \arrow[r, "j_2|_T"] &TH_2(Z,  M_K) \arrow[ur, "\partial |_T"].\\
\end{tikzcd}
\]
Since $H_1(M_K;\Lambda)$ and $H_1(Z, M_K;\Lambda)$ are both $\Lambda$-torsion, the former since this holds for all knots, and that latter by assumption, it follows that $H_1(Z;\Lambda)$ must be torsion as well. We will use this later to conclude that $\ker(j_1|_T)= \ker(j_1)$.

Now define $k:=\grk \im(\partial|_T)$, and  let $x_1, \dots, x_k \in TH_2(Z, M_K;\Lambda)$ be elements whose images  $\partial|_T(x_1), \dots, \partial|_T(x_k)$ under $\partial|_T$ generate $\im(\partial|_T)$.
Since $\grk \Lambda^n=n$,  there exist $y_1,\dots, y_n \in TH_2(Z,M_K;\Lambda)$ that generate $\im(j_2^b)$ as a $\Lambda$-module. Let $z_1,\dots, z_{\ell} \in TH_2(Z,M_K;\Lambda)$ generate $\im(j_2|_T)$ as a $\Lambda$-module  for some $\ell \in \mathbb{N}$.

We claim that $x_1, \dots,x_k,  y_1,\dots,y_n, z_1, \dots, z_{\ell}$ generate $TH_2(Z,M_K;\Lambda)$.  Let $a$ be an arbitrary element of $TH_2(Z,M_K;\Lambda)$.  Since $\partial|_T(x_1), \dots, \partial|_T(x_k)$ generate $\im(\partial|_T)$,  there exist $p_1,\dots,p_k \in \Lambda$ such that
\[\partial|_T(a)= \sum_{i=1}^k p_i \partial|_T(p_i x_i)= \partial|_T\Big(\sum_{i=1}^k p_i x_i \Big).\]
Therefore $b:= a- \sum_{i=1}^k p_i x_i$ is an element of $\ker(\partial|_T)$.  Since $b \in TH_2(Z, M_K)$,  we have that $\partial(b)= \partial|_T(b)=0$,  and so $b$ is an element of \[\ker(\partial) \cap TH_2(Z,M_K;\Lambda)= \im(j_2) \cap TH_2(Z,M_K;\Lambda).\]
Now we assert that $\im(j_2) \cap TH_2(Z,M_K;\Lambda) \subseteq \im(j_2^b) + \im(j_2|_T)$.  Assuming this, we can write $b= \sum_{i=1}^n q_i y_i + \sum_{i=1}^{\ell} r_i z_i$ for some $q_i, r_i \in \Lambda$,  thereby establishing that $x_1, \dots,x_k,  y_1,\dots,y_n, z_1, \dots, z_{\ell}$ generate $TH_2(Z,M_K;\Lambda)$.  To complete the proof of this claim we argue that the assertion \[\im(j_2) \cap TH_2(Z,M_K;\Lambda) \subseteq \im(j_2^b) + \im(j_2|_T)\] holds.  Let $b \in \im j_2$, that is $b= j_2(c) = j_2(c_1,c_2)$ for $c_1 \in \Lambda^n$ and $c_2 \in TH_2(Z;\Lambda)$. More precisely, $b =j_2(c_1,c_2) = j_2^a(c_1) + j_2^b(c_1) + j_2|_T(c_2)$. If in addition $b \in TH_2(Z,M_K;\Lambda)$, then $j_2^a(c_1)=0$ and so indeed $b \in \im (j^b_2) + \im (j_2|_T)$.

It follows that the equivalence classes of $x_1, \dots,x_k, y_1,\dots,y_n$ generate $TH_2(Z,M_K;\Lambda)/ \im(j_2|_T)$,  and hence that $\grk \coker(j_2|_T) \leq n+k$.  By Lemma~\ref{lem:kercoker}, using the hypothesis that $H_1(Z,M_K; \Lambda)$ is torsion in order to apply the lemma,  this implies that $\grk \coker(j_2|_T)  = \grk \ker(j_1|_T)$, and therefore we have
\[\grk \ker(j_1) = \grk \ker(j_1|_T) = \grk \coker(j_2|_T) \leq n+k.  \]
For the first equality we used that $H_1(Z;\Lambda) = TH_1(Z;\Lambda)$, as observed above.
Also note that \[\ker(i_1)= \im(\partial) = \im(\partial^a)+ \im(\partial|_T),\]  and so
\[ \grk \ker(i_1) \leq \grk \im(\partial^a)+ \grk \im(\partial|_T) \leq n+k.\]
Now combine  Proposition~\ref{prop:genrank}~\eqref{item:map}, exactness, and the previous two inequalities $\grk \ker(j_1) \leq n+k$ and $\grk \ker(i_1) \leq n+k$, to obtain
\begin{align*}
\grk \A(K) &= \grk H_1(M_K;\Lambda) \leq \grk \ker(i_1)+  \grk \im(i_1)= \grk \ker(i_1)+  \grk \ker(j_1) \\ &\leq (n+k) + (n+k) = 2n+2k.
\end{align*}
We therefore have that \[\grk \im(\partial|_T) = k \geq (\grk \A(K) -2n )/2 = \smfrac{1}{2}\grk \A(K)- n,\] as desired.
\end{proof}

\subsection{Proofs of Theorems \ref{thm:intro-arb-large} and \ref{thm-main-obstruction-intro}}

We are now ready to prove these two theorems.

\begin{proof}[Proof of Theorem~\ref{thm-main-obstruction-intro}]
Recall that  $(K,\tau)$ is a strongly invertible knot and $k$ is by definition the maximal generating rank of any submodule $P$ of $\A(K)$ satisfying
$\Bl_K(x, y)=0=\Bl_K(x, \tau_*(y))$ for all $x,y \in P$.
Now suppose that $K$ bounds a genus $g$ surface $F$ in $D^4$ such that the involution $\tau$ on $S^3$ extends to an involution $\widehat{\tau}$ on $D^4$ such that $\widehat{\tau}(F) = F$.  We wish to show that $g \geq \frac{ \grk \A(K)-2k}{4}$.

Let $Z$ be as in Proposition~\ref{prop:factsaboutZ},  and consider the following portion of the long exact sequence of $(Z, M_K)$ with $\Lambda$-coefficients:
\[ \cdots  \to H_2(Z, M_K; \Lambda) \xrightarrow{\partial} H_1(M_K; \Lambda) \xrightarrow{i_*} H_1(Z; \Lambda) \to \cdots.\]
Define  $Q:= \partial(TH_2(Z,M_K; \Lambda)) \subseteq H_1(M_K; \Lambda)$.

Our first claim is that $\Bl_K(x,y)=0=\Bl_K(x, \tau_*(y))$ for all $x,y \in Q$.  So let $x,y \in Q$ be given.  Since $Q \subseteq \im(\partial|_T) \subseteq \ker(i_*)$,  Proposition \ref{prop:vanishingofbl} implies that $\Bl_K(x,y)=0$.
Additionally,  $y \in \ker(i_*)$  implies that $\tau_*(y) \in \ker(i_*)$ as well by Proposition~\ref{prop:factsaboutZ}~\eqref{item:kernelinvariant}. Thus $\Bl_K(x,\tau_*(y))=0$  too.   We conclude that $\grk Q \leq k$, by definition of~$k$.

By Proposition~\ref{prop:factsaboutZ}~\eqref{item:relh1lambda} and~\eqref{item:h2zrank},  we have that $H_1(Z, M_K; \Lambda)$ is torsion and the free part of $H_2(Z; \Lambda)$ has rank $2g$.
Therefore, Proposition~\ref{prop:nontrivialitypartial} implies that $\grk Q \geq \frac{1}{2} \grk \A(K) - 2g$.
We therefore have that
\[ k \geq \grk Q \geq \smfrac{1}{2} \grk \A(K) - 2g\]
or,  rewriting,
\[ g \geq \smfrac{\grk \A(K) - 2k}{4}. \qedhere \]
\end{proof}

Finally, Theorem~\ref{thm:intro-arb-large} is an immediate consequence of the following slightly stronger result.

\begin{theorem}
Let $J_1, \dots, J_n$  denote genus one strongly invertible knots with pairwise distinct and nontrivial Alexander polynomials.  Pick a strong inversion $\tau_i$ on $J_i$ for each $i=1, \dots, n$ and choose $a_1,\dots, a_n \in \mathbb{N}$.
 Letting $\#^{a_i} (J_i,  \tau_i )$ denote
 the $a_i$-fold connected sum of $(J_i, \tau_i)$,  define
$(J, \tau) := \#_{i=1}^n (\#^{a_i} (J_i,  \tau_i)).$
Then the equivariant 4-genus of $(J, \tau)$ is at least $\smfrac{1}{4}\max(a_1,\dots,a_n)$.
\end{theorem}

\begin{proof}
First, observe that
\[ \grk \A(J)= \grk \bigoplus_{i=1}^n \A(J_i)^{a_i}
= \grk \bigoplus_{i=1}^n\left( \Q[t^{\pm1}]/ \Delta_{J_i}(t)\right)^{a_i}=\max\{a_1, \dots, a_n\},\]
where the last equality uses the fact that $\Delta_{J_1}(t), \dots, \Delta_{J_n}(t)$ are pairwise distinct,  degree 2,  and symmetric,  hence pairwise relatively prime.

It remains to show that the only element $x \in \A(J)$ with $\Bl_J(x, \tau_*(x))=0$ is the trivial element,  and our result will follow by Theorem~\ref{thm-main-obstruction-intro}.  So write $x=(x_i)_{i=1}^n$,  where each $x_i \in \A(\#^{a_i}J_i)$,  and observe that we can write
$\Bl_{\#^{a_i}J_i}(x_i, (\tau_i)_*(x_i))= \smfrac{p_i(t)}{\Delta_{J_i}(t)}$ for some $p_i(t) \in \Q[t^{\pm1}]$.
So
\[ \Bl_J(x, \tau_*(x))= \sum_{i=1}^n \Bl_{\#^{a_i}J_i}(x_i, (\tau_i)_*(x_i))
= \sum_{i=1}^n \smfrac{p_i(t)}{\Delta_{J_i}(t)}.\]
Since all the $\Delta_{J_i}(t)$ are relatively prime,  this expression is trivial in $\Q(t)/\Q[t^{\pm1}]$  only when $\Bl_{\#^{a_i}J_i}(x_i, (\tau_i)_*(x_i))=\smfrac{p_i(t)}{\Delta_{J_i}(t)}$ vanishes for all $i=1, \dots, n$.  But by Proposition~\ref{prop:genusone} applied to each $J_i$,  this occurs only when $x_i=0$ for all $i=1, \dots, n$, that is when $x=0$.    Therefore $k=0$ in Theorem~\ref{thm-main-obstruction-intro}, so
\[\wt{g}_4(J) \geq \smfrac{\grk \A(J)}{4} = \smfrac{1}{4} \max(a_1,\dots,a_n), \]
as desired.
\end{proof}

\def\MR#1{}
\bibliography{bib}

\end{document}